\documentclass[10pt,a4paper]{article}
\usepackage[utf8]{inputenc}
\usepackage[english]{babel}
\usepackage{multicol,graphicx,color}
\usepackage{pslatex}
\usepackage{authblk}
\usepackage{amsthm}
\usepackage{amsmath}
\usepackage{amssymb}
\usepackage{latexsym}
\usepackage{lscape}
\usepackage{epsfig}
\usepackage{amsfonts}
\usepackage{mathrsfs,bbm}
\usepackage[
hmarginratio={1:1},     
vmarginratio={1:1},     
textwidth=15cm,        
textheight=21cm,
heightrounded,          
]{geometry}

\usepackage{tikz}

\makeatletter

\newtheorem{theorem}{Theorem}[section]
\newtheorem{corollary}[theorem]{Corollary}
\newtheorem{definition}[theorem]{Definition}

\newtheorem{lemma}[theorem]{Lemma}
\newtheorem{proposition}[theorem]{Proposition}
\newtheorem{remark}[theorem]{Remark}

\theoremstyle{definition} \theoremstyle{remark}
\numberwithin{equation}{section}

\definecolor{my_green}{RGB}{20, 200, 20}

\definecolor{my_red}{RGB}{255, 0, 0}

\definecolor{my_purple}{RGB}{255, 76, 135}

\DeclareMathOperator{\morse}{m}

\newcommand{\M}{S_\mu}

\newcommand{\ess}{{\rm ess}}

\newcommand{\eps}{\varepsilon}

\DeclareMathOperator{\supp}{supp}

\newcommand{\loc}{{\mathrm{\textup{loc}}}}

\newcommand{\essinf}{\mathrm{essinf}}
\newcommand{\rad}{\mathrm{rad}}

\newcommand{\R}{\mathbb{R}}
\newcommand{\N}{\mathbb{N}}

\begin{document}

\author{Pablo Carrillo \footnote{pablo.carrillo\_martinez@umlp.fr}}  \affil{Universit\'e Marie et Louis Pasteur, CNRS, LmB (UMR 6623), F-25000 Besan\c{c}on, France}

\author{Louis Jeanjean\footnote{louis.jeanjean@univ-fcomte.fr}}  \affil{Universit\'e  Marie et Louis Pasteur, CNRS, LmB (UMR 6623), F-25000 Besan\c{c}on, France}

\title{Normalized solutions to mass supercritical Schr\"{o}dinger equations with radial potentials}
\date{}

\maketitle

\begin{abstract}
We study the stationary nonlinear Schr\"{o}dinger equation 
\begin{equation} \label{01}
-\Delta u+V(x)u+\lambda u=|u|^{q-2}u, \quad u \in H^1(\R^N), \quad N \geq 2,
\end{equation}
where $V \in L^{\infty}(\R^N)$ is a radial potential. 
In the $L^2$-supercritical regime, we show the existence of an explicit $\mu_0 >0$ such that, for any $\mu \in (0, \mu_0)$, Equation \eqref{01} admits two solutions having $L^2$ norm $\mu$. The potential $V$ is not assumed to have a sign, nor a specific behavior at infinity and only a low regularity is required.  Our proof relies on the use of Morse type information, on some spectral arguments, and on a blow-up analysis developed in a radial setting.  
\end{abstract}

\medskip

{\small \noindent \text{Key Words:} Nonlinear Schr\"odinger equations; $L^2$-supercritical; Normalized solutions; Blow-up techniques.}

\medskip

{\small \noindent \text{Acknowledgements:}
This work has been carried out in the framework of the Project NQG (ANR-23-CE40-0005-01), funded by the French National Research Agency (ANR).
P. Carrillo, and L. Jeanjean thank the ANR for its support. Moreover, an initial version of this article was contained in P. Carrillo's thesis. The authors would like to thank the thesis reviewers, P. Gravejat and F. Robert, for their numerous suggestions, which significantly improved the initial version.}

\medskip

{\small \noindent \text{Statements and Declarations:} The authors have no relevant financial or non-financial interests to disclose.}

{\small \noindent \text{Data availability:} Data sharing is not applicable to this article as no datasets were generated or analyzed during the current study.}

\section{Introduction}

In this paper we are interested in the existence of solutions to the nonlinear Schr\"odinger equation with potential in the whole space $\R^N$

\begin{equation}\label{EQ:Rn}
-\Delta u+V(x)u+\lambda u=|u|^{q-2}u,\quad u	\in H^1(\R^N)
\end{equation}

where it is assumed, all along, that $N\geq 2$, $V\in L^\infty(\R^N)$ and $$ 2+\frac{4}{N}<q<2^*.$$ We recall that $2^* = \frac{2N}{N-2}$, if $N \geq 3$ and $2^* = + \infty$, if $N=2$.

Equations of the form of \eqref{EQ:Rn} appear from the study of time dependent nonlinear Schr\"odinger equations $$-i\frac{\partial}{\partial t}\varphi(x,t)=\Delta_x\varphi(x,t)+f(x,\varphi(x,t)),$$
when searching for standing wave solutions of the form 
$$\varphi(x,t)=e^{i\lambda t}u(x),$$
where $u$ is a real function. Here we consider the case of a separable potential - power non-linearity $$f(x,z)=V(x)+|z|^{q-2}z.$$

A notable direction in the study of Equation \eqref{EQ:Rn} consists in fixing the $L^2$ norm of the desired solution, which physically is interpreted as searching for solutions where the mass of the particle is fixed. In this case, solutions to Equation \eqref{EQ:Rn} may be searched through variational methods and in fact arise as critical points of the $C^2$ energy 
functional
\begin{equation}\label{eq:functional Rn}
    E(u)=\frac{1}{2}\int_{\R^N}|\nabla u|^2+\frac{1}{2} \int_{\R^N}V(x)u^2-\frac{1}{q}\int_{\R^N}|u|^q
\end{equation}
subject to the mass constraint 
\begin{equation}
    \label{eq:mass constraint}
H^1_\mu(\R^N)=\{u\in H^1(\R^N)\colon \|u\|_{L^2(\R^N)}=\mu\},
\end{equation}
where $\mu>0$ is the desired norm. Note that the parameter $\lambda$, present in Equation \eqref{EQ:Rn}, comes out here as a Lagrange multiplier and thus is actually an unknown of the problem. 

Whenever $2<q<2+\frac{4}{N}$, it can be readily seen, using the Gagliardo-Nirenberg inequality \eqref{Eq GN}, that the energy functional defined in \eqref{eq:functional Rn} is bounded from below and coercive on $H^1_\mu(\R^N)$, so solutions can be looked for via a global minimization approach. This case is known as the mass subcritical regime and after, some earlier works, \cite{St81,St82}, its treatment started essentially with the works of P.-L. Lions \cite{Lions1984a, Lions1984b}. This case is still under investigation and we would like to point out \cite{IkMi20,IkYa24} for recent sharp results.

Whenever $2+\frac4{N}<q<2^*$ we are in the mass supercritical regime and it may be seen by testing functions
that the Energy functional is unbounded from below in $H^1_\mu(\R^N)$. As a result, one needs to look for other types of critical points such as local minima or critical points at some minimax level, typically lying at a mountain pass level. In this latter situation it is now well known that a central difficulty is the possible unboundedness of the Palais-Smale sequences (PS sequences for short) for $E$ constrained to $H^1_\mu(\R^N)$. Under the assumption that the problem is autonomous, namely that $V$ is a constant, but considering more general nonlinearities this case we first handled in \cite{Je97}. A key idea to tackling this difficulty was the introduction of a natural constraint related to the Pohozaev identity which furthermore allowed to solve additional compactness issues. Since then, many results on autonomous equations or systems have been obtained, still making use of such type of identity, see for example \cite{BaJe18, BaSo17, BaSo19, BiMe21, GoJe18, IkTa19}. For more general nonlinearities we may refer the reader to 
\cite{JeLu20,So20} and for results concerning the Sobolev critical case
see \cite{JeLe22, So20critical, VeYu25, WeWu22}.

Turning back to cases where $V$ is not a constant let us mention the case of trapping potentials, i.e.
$\lim_{|x|\to\infty}V(x)=+\infty.$
In the study of Schr\"odinger systems like in \cite{NoTaVe15}, this condition allowed to regain some compactness. Let us also mention \cite{BeBoJeVi17,JeSo25} dealing with cases where the potential is partially confining.

Recently the case of non-trapping potentials has begun to be studied in the mass supercritical regime and specifically the case where $\lim_{|x|\to\infty}V(x)=0.$ As already anticipated, one of the main issues is recovering enough compactness for the Palais Smale sequences. Inspired by the techniques used in the autonomous cases, most strategies are still based on the used of a Pohozaev type identity which now reads as
\begin{equation}\label{PoG}
\frac{1}{q}\|u\|_q^q=\frac{2}{N(q-2)}\|\nabla u\|_2^2-\frac{1}{q-2}\int_{\R^N} V(x)u^2-\frac{2}{N(q-2)}\int_{\R^N} V(x)u\nabla u\cdot x.
\end{equation}
However, the presence of the additional last term makes the analysis more delicate and requires further assumptions on certain norms of $V(x)$ and of
$W(x):=|x|V(x)$.

To be more specific, in \cite{BMRV} the authors considered the case of non-negative potentials and either $V\in L^{\frac{N}2}(\R^N)$ or $V\in L^\infty(\R^N)$. In the first case a solution to Equation \eqref{EQ:Rn} was obtained for every $\mu >0$ under the assumption that $\|V\|_\frac{N}{2}$ and $\|W\|_N$ are below a certain explicit constant; in the latter case they required $\|V\|_\infty$ and  $\|W\|_\infty$ to be smaller certain thresholds depending on the mass $\mu >0$. These bounds allowed, somehow, to recover the estimates obtained in the autonomous cases. This paper was later complemented by \cite{MRV}, where existence results involving non-positive potentials with $V\in L^s(\R^N)$ for all $s\in[\frac{N}{2},+\infty]$ were derived with similar smallness conditions (depending on the mass constraint $\mu >0$ for all $s\neq \frac{N}2$) on the norms of both $V$ and $W$. In both papers the assumption that the potential has a sign allows to benefit from a comparison between the considered problem and the associated problem at infinity.

Alternatively, the authors of \cite{BQZ} derived a procedure which involved studying the equation on large but \emph{bounded} domains, then passing to the limit to recover $\R^N$.  Whilst also making a Pohozaev identity come into play, this time the required estimates for the potential terms involved additional assumptions on the derivatives of the potential. It is also worth mentioning that in order to recover compactness they present a novel and interesting condition on the derivative of the potential at infinity.

Actually, in all the above mentioned papers, a Pohozaev identity of the type of \eqref{PoG} plays a crucial role at some stage of the proof. In order to be able to deduce information from it, it seems necessary to impose bounds on $W(x)=|x|V(x)$ which leads to imposing a certain speed of decay of the potential $V$. Sometimes  requiring additional regularity on $V$ is also necessary.

At this point let us observe that the set of solutions to Equation \eqref{EQ:Rn} under the mass constraint \eqref{eq:mass constraint}, remains unchanged if we replace the potential $V(x)$ by $V(x)+d$, where $d \in \R$ is arbitrary.
Thus, the assumption $\lim_{|x|\to\infty}V(x)=0$ may be seen as the existence of a limit at infinity for the potential $V(x)$. In the above-mentioned papers, setting this limit to $0$ allowed for simplification and uniformization when presenting the results. Therefore, the required conditions on $V$, such as a sign or a decay condition at infinity, should be understood as conditions on $(V - \lim_{|x|\to\infty}V(x))$ and accordingly for $W(x)$.

Let us also mention that there  do exist other approaches that do not require the use of a Pohozaev identity nor conditions on $V$ at infinity. For example in \cite{PPVV} the authors instead use a Lyapunov-Schmidt reduction technique to derive existence results. While this manages to deal with the problem of a priori estimates and loss of compactness at infinity, it comes at the cost of having to assume the non-degeneracy of the potential at a given point and also it only provides a perturbation result in the sense that if a solution is obtained for small values of $\mu >0$ there is no explicit bound saying how small $\mu >0$ must be. 

Following the above discussion, in this paper we develop, under the assumption that the potential is radial, an approach to find non-perturbation results which avoids the use of a global Pohozaev identity as \eqref{PoG} and the need of higher orders of regularity of the potential or of norm bounds on $W(x)=V(x)|x|$. In fact, we do not need the potential $V$ to have a limit at infinity either.

In order to state our results we introduce some notations.   We consider the self-adjoint operator
$$-\Delta + V \colon H^2_\rad(\R^N) \subset L^2_\rad(\R^N) \to L^2_\rad(\R^N)$$
and we denote by $\sigma(-\Delta + V)$ its spectrum and by $\lambda_1(V)$ the infimum of its spectrum. Namely 
\begin{equation*}
\label{Def-firsteigenvalue}
\lambda_1(V) :=  \inf_{u \in H^1_{\rad}(\R^N)\backslash \{0\}} \frac{\int_{\R^N} |\nabla u|^2 + V(x) |u|^2}{\int_{\R^N} |u|^2}.
\end{equation*}
Additionally, we denote the essential spectrum of $-\Delta + V$ and its infimum respectively by $\sigma_{ess}(- \Delta + V)$ and $ \lambda_{ess}(V)$.

The assumptions on the potential that will be used in the paper are the following:
\begin{enumerate}
\item[(V0)] $V$ is a radial function;
    \item[(V1)] $V\in L^{\infty}(\R^N)$; 
		\item[(V2)]  The infimum of the spectrum $\lambda_1(V)$ is an eigenvalue.
\end{enumerate}

Let $\underline{V}:= \essinf_{x \in \R^N} V(x)$. As already observed, adding some constant to $V$ makes no difference regarding sets of solutions.  In particular the quantity 
\begin{equation*}
\underline{V} - \lambda_{ess}(V) \leq 0
\end{equation*}
is invariant under such operation. Having this in mind we may now state our main existence results.

\begin{theorem}\label{theo:main}
   Assume that $(V0)-(V2)$ holds. Then, there exists an explicit value $\mu_0 >0$ depending only on $N,q$ and $(\underline{V} - \lambda_{ess}(V))$ such that, for all $\mu \in (0, \mu_0),$ there exists a positive radial solution to Equation \eqref{EQ:Rn} under the mass constraint \eqref{eq:mass constraint} lying at a mountain pass level of $E$ restricted to $H^1_\mu(\R^N)$. 
	Moreover, the associated Lagrange multiplier satisfies $\lambda > - \lambda_1(V).$
\end{theorem}

\begin{theorem}\label{theo: local min}
Assume that $(V0)-(V2)$ holds.  Let $\mu_0 >0$ be as in Theorem \ref{theo:main}. For any $\mu \in (0, \mu_0),$ there exists a positive radial solution to Equation \eqref{EQ:Rn} under the mass constraint \eqref{eq:mass constraint} which is a local minimizer of $E$ restricted to $H^1_\mu(\R^N)$, whose energy level is strictly below the one of the solution obtained in Theorem \ref{theo:main}. Moreover, its associated Lagrange multiplier also satisfies $\lambda > - \lambda_1(V).$
\end{theorem}
Note that 
combining Theorem 
\ref{theo:main} and Theorem \ref{theo: local min}, we obtain for any $\mu \in (0, \mu_0)$ the existence of two solutions for Equation $\eqref{EQ:Rn}$ whose $L^2(\R^N)$ norm is exactly $\mu$. 

Before going any further, we would like to make some comments on our assumptions. 
\begin{remark}\label{R1}
    The threshold mass value $\mu_0 >0$ is explicit and is given in \eqref{eq mu infinity}. Roughly speaking, it appears when proving that, for $\mu>0$ sufficiently small, the functional $E$ restricted to ${H_\mu^1(\R^N)}$ admits a geometry of local minimum. As we shall succeed proving that a local minimum exists this will prove Theorem \ref{theo: local min}. Since we are in a mass supercritical case, the existence of a geometry of local minima implies that $E$ has a geometry of mountain pass (a MP geometry for short). As it will be clear from our proofs, the limitation $\mu \in (0, \mu_0)$ is only used to prove these geometries and not to handle compactness issues. 
\end{remark}

\begin{remark}\label{R2}
  Assumption $(V1)$ implies that solutions to Equation $\eqref{EQ:Rn}$ are in $C^{1,\alpha}_\loc(\R^N)\cap W^{2,\infty}_\loc(\R^N).$ Such regularity result is explicitly used in Lemma \ref{Lemma maxima bounds} below and moreover is sufficient to use strong maximum principles 
	or a comparison principle for Schr\"odinger operators in the proof of Theorem \ref{theo sec blow}.
\end{remark}

\begin{remark}\label{R3}
Assumption $(V2)$ holds if, for example, there exists a
$\varphi_0 \in H^1_{\rad}(\R^N) \backslash \{0\}$ such that
$$\frac{\int |\nabla \varphi_0|^2+ V\varphi_0^2}{\int|\varphi_0|^2}\leq \lambda_{ess}(V).$$ Indeed, if the inequality is strict then $\lambda_1(V) < \lambda_{ess}(V)$ and in that case \cite[Th XIII.48]{ReSi78} yields that $\lambda_1(V)$ is achieved by a positive eigenfunction. Otherwise we have that $\lambda_1(V) = \lambda_{ess}(V)$ and such $\varphi_0 \in H^1_{\rad}(\R^N) \backslash \{0\}$ is a minimizer for $\lambda_1(V)$. 
\end{remark}

\begin{remark}\label{R4}
 If one assumes that $ \lim_{|x| \to \infty} V(x) =0$, then $\lambda_{ess}(V)=0$. Indeed it is straightforward to check that the operator
		
\begin{align*}
        \mathcal{V}\colon H^2_\rad(\R^N)&\to L^2_\rad(\R^N)\\
        \psi&\mapsto V\psi,
    \end{align*}
		is relatively $(-\Delta)$-compact (the definition can be found in \cite[Definition 14.1]{HiSi}) and then, see \cite[Theorem 14.6]{HiSi}, that $\sigma_\ess(-\Delta +V) = \sigma_\ess(-\Delta )= [0, + \infty).$ However this is only a sufficient condition and $\lambda_{ess}(V) =0$ may be true without the existence of a limit for $V$ at infinity. 
\end{remark}

\begin{remark}\label{R5}
If one assumes that $ \lim_{|x| \to \infty} V(x) =0$, then $\lambda_1(V) \leq 0$ and for Assumption $(V2)$ to hold the negative part of $V$ must be non-trivial. In particular, for $N\geq 3$ the Cwikel-Lieb-Rozenblum theorem tells us that, to have $\lambda_1(V) <0$, a necessary condition is for $\|V_-\|_{\frac{N}{2}}$ to be above a certain positive threshold, see \cite[Theorem 4.31]{FrLaWe23}.  In the case $N=2$, however, where having $\|V_-\|_\infty\neq 0$ is a sufficient for the existence of a negative eigenvalue.

\end{remark}

\begin{remark}\label{R6}
The assumption that the potential $V$ is radial allows to work in the subspace of radial functions $H^1_{\rad}(\R^N)$ of $H^1(\R^N)$. As such, we benefit from the compact inclusion of $L^s(\R^N)$ into $H^1_{\rad}(\R^N)$ for any $s \in (2, 2^*)$, property which is helpful to control the loss of mass at infinity. However, several of our results still hold without this restriction. In turn the fact that we work in $H^1_{\rad}(\R^N)$ will require new arguments in the blow-up analysis that we will perform in order to prove Theorem \ref{theo:main}.
\end{remark}

Turning now to the proof of Theorem \ref{theo:main}, in order to obtain the existence of the required a critical point, it appears necessary to show that, at least one PS sequence for $E$ constrained to ${H_\mu^1(\R^N)}$ at the MP level, is bounded. To do so, we shall borrow a general strategy which was developed in a series of works \cite{BoChJeSo23, CaGaJeTr25, ChJeSo24} devoted to the existence of prescribed $L^2$ norm solutions for Schr\"odinger equations set on metric graphs. The common feature with our problem being the lack of a useful Pohozaev identity due to the impossibility to use scaling on graphs.  As a first step we consider, for $\rho\in [\frac12,1]$, the problem
\begin{equation}\label{EQ:bounded approximation}
\begin{cases}
-\Delta u+V(x)u+\lambda u=\rho|u|^{q-2}u,\quad u \in H^1_\rad(\R^N),\\
\|u\|_{L^2(\R^N)}=\mu,
\end{cases}
\end{equation}
with $N\geq 2$, $V\in L^\infty$, and $2+\frac{4}{N}<q<2^*$, whose associated energy 
functional is now
\begin{equation}
\label{eq:functional approximation}
E_{\rho}(u)=\frac{1}{2}\int_{\R^N}|\nabla u|^2+\frac{1}{2} \int_{\R^N}V(x)u^2-\frac{\rho}{q}\int_{\R^N}|u|^q
\end{equation}
constrained to $$H^1_{\mu,\rad}(\R^N):=\{u\in H^1_\rad(\R^N)\colon \|u\|_{L^2(\R^N)}=\mu\}.$$ 
The underlying idea is to use the monotonicity trick, in terms of the formulation given in \cite{Je99},
in order to obtain, for almost every $\rho\in [\frac12,1]$, a bounded PS sequence at the corresponding MP level. 
Actually we shall directly make use of \cite[Theorem 1.5]{BoChJeSo24}, recalled here in Proposition \ref{lemma:result critical point theory}, which yields the existence of a bounded PS sequence with additional Morse type information.
{Note that since we work in the subspace $H^1_{\rad}(\R^N)$, Proposition \ref{lemma:result critical point theory} provides information on the \emph{radial} Morse index. This information will be centrally used in a blow-up analysis. In the following, for a function $u \in H^1_{\rad}(\R^N)$, $\morse(u)$ will always denote the radial Morse index. We refer to Definition \ref{def Morse index} for a precise definition. }

Note that to ensure that the limit of this PS sequence belongs to the mass constraint, we must prove its strong convergence. This is one of the central difficulties in the search of solutions having a prescribed $L^2$ norm and, in general, the strategy resides in first proving that the Lagrange multiplier associated with the weak limit of the PS sequence, which we will call $\lambda_\rho$, satisfies an appropriate \emph{strict inequality}.
In addition to checking the boundedness of (some particular) PS sequences, it is also at this point that the Pohozaev identity proved particularly useful in many cases. Here, we shall use spectral arguments to show that  $\lambda_{\rho} > - \lambda_1(V)$. 
 Having this  information and using that, thanks to the radial nature of the potential, we can work in the subspace of radially symmetry functions we then  deduce the strong convergence and obtain the following result.
\begin{theorem}\label{theo main aproximation}
    Assume that $(V0)-(V2)$ holds and let $\mu \in (0, \mu_0)$ where $\mu_0 >0$ is defined in Theorem \ref{theo:main}. Then, for almost every $\rho\in[\frac12,1]$, there exists a positive radial solution $u_{\rho} \in H^1_{\mu}(\R^N)$ to Problem \eqref{EQ:bounded approximation}. This solution lies at a MP level, denoted by $c_{\rho}$, which is defined in \eqref{eq:MP definition}. It has an associated  Lagrange multiplier strictly above $-\lambda_1(V)$ and {its Morse index $\morse(u_{\rho})$ is at most $2$. }
\end{theorem}

To derive Theorem \ref{theo:main} from Theorem \ref{theo main aproximation} we choose a sequence $\{\rho_n\} \subset [\frac12,1]$ with $\rho_n \to 1^-$ for which Theorem   \ref{theo main aproximation} provides a sequence $\{u_{\rho_n}\} \subset H^1_{\mu,\rad}(\R^N)$ of solutions to \eqref{EQ:bounded approximation} where $\rho = \rho_n$. 
It is readily seen that if such sequence converges (strongly), its limit will be a solution as searched for in Theorem \ref{theo:main}.
Exploiting the analysis performed in the derivation of Theorem \ref{theo main aproximation}, this convergence issue can be reduced first to showing the boundedness of $\{u_{\rho_n}\}\subset H^1(\R^N)$ and then further to proving the boundedness of the associated sequence of Lagrange multipliers $\{\lambda_{\rho_n}\} \subset \R$. This latter condition will be the one proven through means of a blow-up analysis in the spirit of \cite{DHR,EsPe11}. Actually, it will be a consequence of more 
general result of independent interest:
\begin{theorem}\label{theo maxima to 0}
Let $(V0)-(V1)$ hold and {$\{(\rho_n,\lambda_n,u_n)\} \subset [\frac12,1]\times\R\times H_\rad^1(\R^N)$} be a sequence of solutions to
\begin{equation}\label{L_EQ:bounded approximation}
\begin{cases}
-\Delta u+V(x)u+\lambda u=\rho u^{q-1}\\
u>0,
\end{cases}
\end{equation}
 such that $\rho_n\to 1^-$.
Suppose moreover there exist constants $\mu_1,\mu_2>0$ and $c_1\in\R$ such that 
\begin{equation}\label{eq mass energy unif bound}
    \mu_1\leq\|u_n\|_{L^2(\R^N)}\leq \mu_2\quad \quad \text{and}\quad \quad E_{\rho_n}(u_n)\leq c_1
\end{equation}
where $E_{\rho_n}$ is the associated energy functional defined in \eqref{eq:functional approximation}.
Suppose finally that the Morse index of the sequence $\{u_n\}$ is uniformly bounded, i.e. there exists $m^*<+\infty$ such that,
\begin{equation*}
\sup_{n\in \N} \morse(u_n) \leq m^*.
\end{equation*}
Then, \begin{equation*}
    \sup_{n\in \N} \lambda_n < +\infty.
\end{equation*}
\end{theorem}
To prove Theorem \ref{theo maxima to 0} we shall proceed by contradiction and assume that $\lambda_n \to + \infty$.  Under this assumption it is readily seen, see Lemma \ref{Lemma maxima bounds}, that the sequence $\{u_n\} \subset L^{\infty}(\R^N)$ is necessarily unbounded and we develop a blow-up analysis that will be used to find a contradiction. This blow-up analysis, performed in Theorems \ref{mainblow} and \ref{theo sec blow}, see also Corollary \ref{all maximum points}, leads to the conclusion that, passing to the subsequence, a blow-up phenomenon is only possible either at the origin or at a finite number of radii. In addition, the solutions exponentially decay away from those blow-up points. Moreover, in order to obtain these intermediate results one may remark that Assumption \eqref{eq mass energy unif bound} has not been yet utilised. In fact, it only plays a role at the end of the proof of Theorem \ref{theo maxima to 0} to conclude that no blow-up may happen.

Compared to the blow-up analysis that is also performed in \cite{BoChJeSo23, CaGaJeTr25, ChJeSo24}, the fact that we work here in the subspace of radially symmetric functions presents new challenges. 
As $\lambda_n\to +\infty$, it is rather standard to have solutions of equations of the type of \eqref{L_EQ:bounded approximation} concentrating at a singular point, leading to the formation of the so called spike-layer solutions, or simply spikes. In the case of radial functions it is clear that spikes may only form at the origin. However, still in a radial context, if a sequence of local maximum does not converge to the origin what we actually obtain are solutions which concentrate around a sphere, phenomenon which is much less understood. First attempts to study concentration on higher dimensional sets can be found in \cite{AmMaNi03-1} and \cite{dPKoWe06} where conditions on the potential were derived in order to show the existence of radial solutions concentrating around spheres and curves for the related equation
\begin{equation*}
    -\epsilon\Delta u+V(x)u=u^p,\quad u>0, \quad p >1,
\end{equation*}
in $\R^N$ as $\epsilon\to 0.$ When it comes to solutions with (uniformly) bounded Morse index, the number of concentration sets usually remains bounded. This phenomenon is highlighted in \cite{EMSS}, see also \cite{EsPe11,PiVe17} for a related result for spikes. In \cite{EMSS}, under this additional condition, the authors consider the problem
\begin{equation}\label{LL_EQ:bounded approximation}
\begin{cases}
-\Delta u+ \lambda V(x)u = u^q \quad \mbox{in } \Omega\\
u>0, \quad \mbox{in } \Omega\\
u=0, \quad \mbox{on } \partial \Omega,
\end{cases}
\end{equation}
where $q >1$, $V$ is smooth, and $\Omega \subset \R^N$ is an annulus. In particular they explicit necessary conditions on an auxiliary potential for a blow-up to take place at a given radius. Let us also mention the work \cite{LiSo24} where \eqref{LL_EQ:bounded approximation} is considered in the special case $V \equiv 1$. There the focus is, as in our case, on the existence of multiplicity of normalized solutions. In our setting, even if we do benefit from some ideas introduced in \cite{EMSS}, the fact that Equation \eqref{EQ:Rn} is set on the full space $\R^N$ induces new features. Indeed, we must also consider the possibility that sequences of blow-up points escape to infinity or, on the contrary, accumulate around the origin. In that later case, one needs to distinguish two rates of accumulation leading to the analysis of distinct limit problems.
Having proved Theorems \ref{mainblow} and \ref{theo sec blow}, the fact that the additional assumption \eqref{eq mass energy unif bound} implies the boundedness of $\{\lambda_n\}$ is derived making use of a one-dimensional Pohozaev type identity, see Lemma \ref{one Pohozaev}.

\smallskip

The outline of the paper is the following. In Section \ref{section: preliminaries} we recall the necessary definitions and results that we shall need. In Section \ref{section:MP} 
we prove Theorem \ref{theo main aproximation} and Theorem \ref{theo:main} by supposing true Theorem \ref{theo maxima to 0}. Finally, in Section \ref{section: blow up} we perform the blow-up analysis which enables the derivation of Theorem \ref{theo maxima to 0}.
\medskip

{\bf Notations:} Throughout the article we use $\|\cdot\|_s$ to denote the usual $L^s$ norm when the function space is evident. 
For $a \geq 0$ and $R\geq 0$, we define
$$A(a,R) := \{x \in \R^N\colon a-R\leq |x| \leq a +R\}.$$
When $a>R$ the set is an annulus, whilst when $a\leq R$ the set becomes a ball; for the latter case we shall often use the notation $B(0,a + R) = A(a,R)$.
Finally, $\R^+$ will denote the set of strictly positive real numbers. 

\medskip

\section{Preliminaries}\label{section: preliminaries}

For $s\in (2,2^*),$ we denote by $C_{N,s}$ the optimal constant in the Gagliardo-Nirenberg inequality, i.e. the optimal constant only depending on $N$ and $s$ such that

\begin{equation}\label{Eq GN}
\|u\|_s^s\leq C_{N,s}\|u\|_2^{\theta_s}\|\nabla u\|_2^{\xi_s}
\end{equation}
for all $u\in H^1(\R^N)$, where
$$\theta_s:=\frac{2s-N(s-2)}{2}\quad \mathrm{and} \quad \xi_s:=\frac{N(s-2)}{2}.$$

We now introduce the definition of the Morse index used in the statement of Theorems \ref{theo main aproximation} and \ref{theo maxima to 0}. 
\begin{definition}\label{def Morse index}
Let $N\geq 2$ and let $u\in H_\loc^1(\R^N)$ be a radial weak solution to
$$-\Delta u + V(x) u + \lambda u = \rho |u|^{q-2}u,$$
where $V \in L_{\rad}^{\infty}(\R^N)$, $\lambda, \rho \in \R$ and $2<q <2^*$.
 The \emph{Morse index} of $u$, denoted $\morse(u)$, is the maximum
  integer $m$ such that there exists a subspace
  $\mathcal{V}\subset H^1_\rad(\R^N) \cap C_{c}^\infty(\R^N)$ of dimension $m$ with the property that for all $ \varphi \in \mathcal{V} \setminus\{0\}$
  there holds
  \begin{align}
Q_{V}^N(\varphi;u)&:= \int_{\R^N} |{\nabla\varphi}|^2 + (V(x)+\lambda - \rho(q-1)|u|^{q-2}
)|\varphi^2|   dx<0.\label{second differential no beta}
\end{align}
In the case that $V\equiv 0$ we will use the notation $Q_0^N(\varphi;u)$.
\end{definition}

In our proofs, we will have to consider limit problems posed on the real line. In that case we shall make use of the following definition of the Morse index.
\begin{definition}\label{def Morse index in one dimension}
Let $u\in H_\loc^1(\R)$ be a  weak solution of
$$-u''+V(r)u+\lambda u=\rho |u|^{q-2}u.$$
The \emph{Morse index} of $u$, still denoted $\morse(u)$, is the maximum
  integer $m$ such that there exists a subspace
  $\mathcal{V}\subset H^1(\R) \cap C_{c}^\infty(\R)$ of dimension $m$ with the property that for all $ \varphi \in \mathcal{V} \setminus\{0\}$
  there holds
  \begin{align}
Q_{V}^1(\varphi;u)&:= \int_{\R} |{\varphi'}|^2 + (V(x)+\lambda - \rho(q-1)|u|^{q-2}
)|\varphi^2|   dx<0.\label{second differential no beta 1D}
\end{align}
As before, in the case that $V\equiv 0$ we will use $Q_0^1(\varphi;u)$.
\end{definition}

We now turn to the abstract result from critical point theory \cite[Theorem 1.5]{BoChJeSo24}. In order to introduce it we need to put ourselves into some context. Let $\left(E,\langle \cdot, \cdot \rangle\right)$ and $\left(H,(\cdot,\cdot)\right)$ be two {\emph{infinite-dimensional}} Hilbert spaces and assume that $E\hookrightarrow H \hookrightarrow E'$,
with continuous injections.  For simplicity, we assume that the continuous injection $E\hookrightarrow H$ has norm at most $1$ and identify $E$ with its image in $H$.
Set \[ \begin{cases} \|u\|^2=\langle u,u \rangle,\\ |u|^2=(u,u),\end{cases}\quad u\in E,\]
and define for $\mu>0$ the set \[S_\mu= \{ u \in E\big|\, |u|^2=\mu^2 \}. \]
In the context of this paper, we shall have $E=H^1_\rad(\R^N)$, $H=L^2_\rad(\R^N)$ and $$S_\mu=H_{\mu,\rad}^1=\{u\in H^1_\rad:\|u\|_{L^2(\R^N)}=\mu\}.$$
Clearly, $S_{\mu}$ is a smooth submanifold of $E$ of codimension $1$. Furthermore its tangent space at a given point $u \in S_{\mu}$ can be considered as the closed subspace of codimension $1$ of $E$ given by:
\[  T_u S_{\mu}= \{v \in E \big|\, (u,v) =0 \}.\]
{ In the following definition, we denote by $\|\cdot\|_*$ and $\|\cdot\|_{**}$, respectively, the operator norm of $\mathcal{L}(E,\R)$ and of $\mathcal{L}(E,\mathcal{L}(E,\R))$.}

\begin{definition}\label{Holder-continuous}
Let $\phi : E \rightarrow \mathbb{R}$ be a $C^2$-functional on $E$ and $\alpha \in (0,1]$. We say that $\phi'$ and $\phi''$ are $\alpha$-H\"older continuous on bounded sets if for any $R>0$ one can find $M=M(R)>0$ such that, for any $u_1,u_2\in B(0,R)$:
\begin{equation*}
||\phi'(u_1)-\phi'(u_2)||_*\leq M ||u_1-u_2||^{\alpha}, \quad ||\phi''(u_1)-\phi''(u_2)||_{**} \leq M||u_1-u_2||^\alpha.
\end{equation*}
\end{definition}

\begin{definition}\label{def D}
Let $\phi$ be a $C^2$-functional on $E$. For any $u\in E$, we define the continuous bilinear map:
\[ D^2\phi(u)=\phi''(u) -\frac{\phi'(u)\cdot u}{|u|^2}(\cdot,\cdot). \]
\end{definition}

\begin{definition}\label{def: app morse_2}
Let $\phi$ be a $C^2$-functional on $E$. {For} any $u\in \M$ and $\theta >0$, we define
 the \emph{approximate Morse index} by
\[
\tilde m_\theta(u)= \sup \left\{\dim\,L\left| \begin{array}{l} \ L \text{ is a subspace of $T_u \M$ such that: }
D^2\phi(u)[\varphi, \varphi]<-\theta \|\varphi\|^2, \quad {\forall \varphi \in L \setminus \{0\}} \end{array}\right.\right\}.
\]
If $u$ is a critical point for the constrained functional $\phi|_{\M}$ and $\theta=0$, we say that this is the \emph{Morse index of $u$ as constrained critical point}.
\end{definition}

\begin{remark}\label{rem: morse connexions}
    Suppose $u$ is a critical point of the functional $\phi|_{S_\mu}.$ Then $D^2\phi(u)$ restricted to $T_uS_\mu$ coincides with the constrained Hessian of $\phi_{S_\mu}$ at $u$. 
\end{remark}

We now recall \cite[Theorem 1.5]{BoChJeSo24}.

\begin{proposition}\label{lemma:result critical point theory}
Let $I\subset (0,+\infty)$ be an interval and consider a family of $C^2$ functionals $\Phi_\rho\colon E\to \R$ of the form 
$$\Phi_\rho(u)=A(u)-\rho B(u),\quad \quad \rho\in I,$$
where $B(u)\geq 0$ for all $u\in E$, and 
\begin{equation}\label{hp coer}
\text{either $A(u) \to +\infty$~ or $B(u) \to +\infty$ ~ \quad as $u \in E$ and $\|u\| \to +\infty$.}
\end{equation}
Suppose moreover that 
$\Phi'_\rho$ and $\Phi''_\rho$ are $\alpha$-H\"older continuous on bounded sets in the sense of Definition~\ref{Holder-continuous} for some $\alpha \in (0,1]$.
Finally, suppose there exist $w_0,w_1 \in  S_\mu$ (independent of $\rho$) such that, setting
\begin{equation*}
\Gamma=\{ \gamma\in C([0,1], S_\mu)\colon \gamma (0) = w_0, \gamma(1) = w_1 \},
\end{equation*}
we have 
\begin{equation*}
c_{\rho}=\inf_{\gamma\in \Gamma}\max_{t\in [0,1] }\Phi_\rho(\gamma(t))>\max\big\{\Phi_\rho (w_0), \Phi_\rho (w_1)\big\}, \quad \forall \rho \in I.
\end{equation*}
     Then, for almost every $\rho \in I$, there exist sequences $\{u_n\} \subset \M$ and $\zeta_n \to 0^+$ such that, as $n \to + \infty$,
    \begin{itemize}
        \item[(i)] $\Phi_\rho(u_n)\to c_\rho$;
        \item[(ii)] $\|\Phi_\rho'|_{S_\mu}(u_n)\|\to 0$;
        \item[(iii)] $\{u_n\}$ is bounded in $E$;
        \item[(iv)] $\tilde m_{\zeta_n}(u_n)\leq 1$.
    \end{itemize}
\end{proposition}
\begin{remark}\label{strong-convergence}
  If the sequence $\{u_n\} \subset \M$ provided by
 Proposition \ref{lemma:result critical point theory} converges to some $u_\rho \in \M$, then
  in view of points (i)--(ii),
  $u_\rho$ is a critical point of
  $\Phi_\rho|_{\M}$ at level $c_{\rho}$.  Let us show that the
  Morse index of $u_{\rho}$, as a
  constrained critical point, satisfies
  $\tilde m_{0}(u_{\rho}) \leq 1$.  Assume by contradiction
  that this is not the case. Then, in view of Definition \ref{def: app
    morse_2}, we may assume that there exists a subspace
  $W_0 \subset T_{u_{\rho}}\M$ with $\dim W_0 =2$ such that
  \begin{equation*}
    D^2 \Phi_\rho(u_{\rho})[w,w]<0
    \quad \text{for all } w \in W_0 \setminus \{0\}.
  \end{equation*}
  Since $W_0$ is of finite dimension, its unit sphere is
    compact and therefore there exists $\theta > 0$ such that
  \begin{equation*}
    D^2 \Phi_\rho(u_{\rho})[w,w] < -\theta \|w\|^2
    \quad \text{for all } w \in W_0 \setminus\{0\}.
  \end{equation*}
  Now, 
  using that
  $\Phi_\rho'$ and $\Phi_\rho''$ are $\alpha$-H\"older continuous on
  bounded sets for some $\alpha \in (0,1]$, it directly follows that
  there exists $\delta>0$ small enough such that for any $v\in\M$
  satisfying $\|v-u_{\rho}\| \leq \delta$ there holds
  \begin{equation*}
    D^2\Phi_\rho(v)[w,w] < -\frac{\theta}{2} \|w\|^2
    \quad \text{for all } w \in W_0 \setminus \{0\}.
  \end{equation*}
  In particular, for $n$ large enough we have $\|u_n - u_\rho\| \le
    \delta$ and $\zeta_n < \theta/2$ (for a given sequence $\zeta_n \to 0^+$),
    so the previous inequality implies
  \begin{equation*}
    D^2\Phi_\rho(u_n)[w,w] < -\frac{\theta}{2}\|w\|^2 < -\zeta_n\|w\|^2
    \quad \text{for all } w \in W_0 \setminus \{0\} .
  \end{equation*}
  From Proposition \ref{lemma:result critical point theory} (iv), the above directly implies that $\dim W_0\leq 1$ reaching the desired contradiction.

\end{remark}

\begin{remark}\label{remark almost lagrange}
As a consequence of Proposition \ref{lemma:result critical point theory}  (ii)--(iii), we deduce in a standard way, see \cite[Remark 1.3]{BoChJeSo24} or \cite[Lemma 3]{BeLi2}, that 
\begin{equation*}
\label{free-gradient}
  \Phi'_{\rho}(u_n) + \displaystyle \lambda_n (u_n, \cdot) \to 0
  \quad\text{ in }  E' \text{ as } n \to + \infty
\end{equation*}
where we have set 
\begin{equation}\label{def-almost-Lagrange} 
  \lambda_n
  := - \displaystyle \frac{1}{\mu} \Phi'_\rho(u_n)[u_n].
\end{equation} 
We call the sequence $\{\lambda_n\} \subset \R$ defined in
\eqref{def-almost-Lagrange} the sequence of \emph{almost Lagrange
  multipliers}. \medskip
\end{remark}

We finish the section with a result that will allow to derive information on the sequence of almost Lagrange multipliers. It is a direct extension of \cite[Lemma 2.7]{CaGaJeTr25}, see also \cite[Lemma 2.5]{BoChJeSo23}.
\begin{lemma}\label{lemma:Louis1}
  Let $\{u_n\} \subset S_{\mu}$, $\{\lambda_n\} \subset \mathbb{R}$
  and $\{\zeta_n\}\subset \mathbb{R}^+$ with $\zeta_n \to 0^+$.
	Assume
  that, for a given $M \in \N$, the following conditions hold:
  \begin{itemize}
  \item[(i)] For large enough $n \in \mathbb{N}$, all subspaces
    $W_n \subset E$ with the property
    \begin{equation}\label{LL-Hess crit*}
      \Phi''_{\rho}(u_n)[\varphi,\varphi] + \lambda_n|\varphi|^2
      < -\zeta_n \|\varphi\|^2,
      \qquad \text{for all }\, \varphi \in W_n \setminus \{0\},
    \end{equation}
    satisfy that $\dim (W_n) \leq M$.
  \item[(ii)] There exist $\lambda \in \mathbb{R}$, a subspace $Y$ of
    $E$ with $\dim (Y) \geq M+1$ and $a > 0$ such that, for large enough
    $n \in \mathbb{N}$,
    \begin{equation}\label{L1*}
      \Phi''_{\rho}(u_n)[\varphi,\varphi] + \lambda |\varphi|^2
      \le -a \|\varphi\|^2,
      \qquad \text{for all }\, \varphi \in Y.
    \end{equation}
  \end{itemize}
  Then $\lambda_n > \lambda$ for all large enough $n \in
  \mathbb{N}$. In particular, if (\ref{L1*}) holds for any
  $\lambda < \lambda_0$, for some $\lambda_0 \in \R$, then
  $\displaystyle \liminf_{n \to \infty} \lambda_n \geq \lambda_0.$
\end{lemma}
\begin{proof}
  Suppose by contradiction that $\lambda_n \leq \lambda$ along a
  subsequence still denoted $\{\lambda_n\}$.  Keep
    denoting $\{u_n\}$ and $\{\zeta_n\}$ the corresponding subsequences.
  From \eqref{L1*} we have,
\begin{equation*}
  \Phi''_{\rho}(u_n)[\varphi,\varphi]+\lambda_n|\varphi|^2
  \le \Phi''_{\rho}(u_n)[\varphi,\varphi] + \lambda |\varphi|^2
  \le -a \|\varphi\|^2
  \qquad
  \text{for all } \varphi \in Y \setminus \{0\}.
\end{equation*}
Now, since $\zeta_n\to 0^+$, there exists $n_0 \in \mathbb{N}$ such that: $\forall n \geq n_0$,
$\zeta_n < a$. Thus, for an arbitrary $n \geq n_0$, we obtain 
\begin{equation*}
  \Phi''_{\rho}(u_n)[\varphi,\varphi]+\lambda_n|\varphi|^2
  \le  -\zeta_n \|\varphi\|^2,
  \qquad \text{for all } \, \varphi \in Y \setminus \{0\},
\end{equation*}
which is in contradiction with \eqref{LL-Hess crit*} since $ \dim  (Y) \geq M+1$. 
\end{proof}

\section{Proof of Theorem \ref{theo main aproximation}, Theorem \ref{theo: local min} and conditional proof of Theorem \ref{theo:main}.}\label{section:MP}

In this section we first give the proof of Theorem \ref{theo main aproximation} and of Theorem \ref{theo: local min}. Next, assuming that Theorem \ref{theo maxima to 0} holds we give the proof of Theorem \ref{theo:main}. 
\smallskip

In the proof of Theorem \ref{theo main aproximation}, Proposition \ref{lemma:result critical point theory} will be a key ingredient. We recall that the family of functionals $E_{\rho}\colon H_{\mu,\rad}^1(\R^N)\to \R$, with $\rho\in [\frac12,1]$, is defined by \begin{equation}\label{Eq aux funct E}
E_{\rho}(u)=\frac{1}{2}\int_{\R^N}|\nabla u|^2+\frac12\int_{\R^N}V(x)|u|^2-\frac{\rho}{q}\int_{\R^N}|u|^q.
\end{equation} 
Clearly $E_{\rho}$ is of the form $E_{\rho}(u) = A(u) - \rho B(u)$ and \eqref{hp coer} holds. It can also be checked, in a standard way, that $E'_\rho$ and $E''_\rho$ are $\alpha$-H\"older continuous on bounded sets in the sense of Definition~\ref{Holder-continuous} for some $\alpha \in (0,1]$. To show that  $E_{\rho}$ enters the framework of Proposition \ref{lemma:result critical point theory}, it remains to show that it has a mountain pass geometry on $H^1_{\mu,\rad}(\R^N)$ uniformly with respect to $\rho\in [\frac12,1]$.

\begin{lemma}\label{lemma: the mountain pass}
 Assume that $(V0)-(V2)$ hold. For every $ \mu \in (0, \mu_0)$, where $\mu_0 >0$ is the unique positive solution to
\begin{equation}\label{eq mu infinity}
\mu_0^{-\frac{4q-2N(q-2)}{N(q-2)-4}}\left(\frac{2q}{N(q-2)C_{N,q}}
\right)^\frac{4}{N(q-2)-4} \left(\frac{N(q-2)-4}{2N(q-2)}\right)+\frac12\mu_0^2 \big(\underline{V}- \lambda_{ess}(V)\big)= 0,
\end{equation} 
 there exist $u_0, u_1 \in H^1_{\mu,\rad}(\R^N)$ independent of $\rho \in \left[\frac12,1\right]$ such that
\begin{equation}\label{eq:MP definition}
c_{\rho}:=\inf\limits_{\gamma\in
\Gamma} \, \max\limits_{t\in[0,1]}E_{\rho}(\gamma(t)) > \frac{1}{2}\mu^2 \lambda_{ess}(V) \geq \frac{1}{2}\mu^2 \lambda_{1}(V) >  \max\{E_{\rho}(u_0), E_{\rho}(u_1)\}, \qquad \forall
\rho\in \left[\frac{1}{2},1\right],
\end{equation}
where
$\Gamma:=\{\gamma\in C([0,1], H^1_{\mu,\rad}(\R^N))\big| ~\gamma~ \mbox{is continuous},~\gamma(0)=u_0,\gamma(1)=u_1\}.$
\end{lemma}
\begin{proof}
Fix an arbitrary $\mu \in (0, \mu_0)$.  In order to prove the lemma it suffices to find a bounded open set $\Sigma_{\mu} \subset H^1_{\mu,\rad}(\R^N)$ and functions $u_0\in \Sigma_{\mu}$ and $u_1\notin \bar\Sigma_{\mu}$  such that
\begin{equation*}
\max\{E_\rho(u_0),E_\rho(u_1)\} < \frac{1}{2}\mu^2 \lambda_1(V) \leq
\frac{1}{2}\mu^2 \lambda_{ess}(V) <
\inf_{u\in\partial\Sigma_{\mu}}E_{\rho}(u), \qquad \forall
\rho\in \left[\frac{1}{2},1\right].
\end{equation*}
In this direction, we may find a lower bound of $E_\rho$ using \eqref{Eq GN}:
\begin{align}
E_\rho(u)&\geq \frac{1}{2}\int_{\R^N}|\nabla u|^2+\frac12\mu^2\underline{V}-\frac{\rho}{q}\int_{\R^N}|u|^q\nonumber \\
&\geq \frac{1}{2}\|\nabla u\|^2_{2} +\frac12\mu^2\underline{V}- \frac{C_{N,q}}{q}\mu^{\theta_q}\|\nabla u\|_{2}^{\xi_q}
\label{eq:energy lower bound}.
\end{align}
Aiming to prove \eqref{eq:MP definition} we consider the real valued function
$$f_{\mu}(t)=\frac12t^2  +\frac12\mu^2(\underline{V} - \lambda_{ess}(V))- \frac{C_{N,q}}{q}\mu^{\theta_q}t^{\xi_q}, \quad t>0.$$
This function has a unique global maximum point $t_{\mu} \in \R$ satisfying
\begin{equation*}
    0<t_{\mu}^{\xi_q-2}=\frac{q}{\xi_qC_{N,q}\mu^{\theta_q}},
\end{equation*}
and 
\begin{equation}\label{eq:energy inf estimate}
\begin{split}
f_{\mu}(t_{\mu})&=t_{\mu}^2\left(\frac12-\frac{C_{N,q}}{q}\mu^{\theta_q}t_\mu^{\xi_q-2}\right) +\frac12\mu^2(\underline{V} - \lambda_{ess}(V))=\left(\frac{q}{\xi_qC_{N,q}\mu^{\theta_q}}\right)^{\frac{2}{\xi_q-2}}\left(\frac{1}{2}-\frac{1}{\xi_q}\right) +\frac12\mu^2 \big(\underline{V} - \lambda_{ess}(V)\big)\\
&=\mu^{-\frac{4q-2N(q-2)}{N(q-2)-4}}\left(\frac{2q}{N(q-2)C_{N,q}}
\right)^\frac{4}{N(q-2)-4}\left(\frac{N(q-2)-4}{2N(q-2)}\right) +\frac12\mu^2\big(\underline{V} - \lambda_{ess}(V)\big).
\end{split}
\end{equation}
Note also that, for any $t>0$, since $\underline{V}- \lambda_{ess}(V) \leq 0$ we have
\begin{equation}\label{L1501}
f_{\mu_1}(t) \geq f_{\mu_2}(t), \quad \mbox{if} \quad 0 < \mu_1 \leq \mu_2.
\end{equation}
Using the estimates \eqref{eq:energy lower bound}, \eqref{eq:energy inf estimate} and the definition of $f_{\mu}$ we obtain
\begin{equation}\label{LL_positive infimum}
\begin{split}
\inf&\Big\{E_{\rho}(u) - \frac{1}{2} \mu^2 \lambda_{ess}(V)  \colon u\in H^1_{\rad}(\R^N),\ \|u\|_2 =\mu, \ \|\nabla u\|_2= t_{\mu}\Big\}\\
& \geq \mu^{-\frac{4q-2N(q-2)}{N(q-2)-4}}\left(\frac{2q}{N(q-2)C_{N,q}}
\right)^\frac{4}{N(q-2)-4} \left(\frac{N(q-2)-4}{2N(q-2)}\right) +\frac12\mu^2(\underline{V} - \lambda_{ess}(V)).
\end{split}
\end{equation}
We set
\begin{equation}\label{defSigma}
\Sigma_{\mu} :=\Big\{u\in H^1_{\rad}(\R^N)\colon\|u\|_2=\mu,\ \|\nabla u\|_2< t_{\mu}\Big\}
\end{equation}
and observe, from the definition of $\mu_0 >0$ and \eqref{LL_positive infimum}, that
\begin{equation}\label{L_positive infimum}
\inf\Big\{E_{\rho}(u) - \frac{1}{2}\mu^2 \lambda_{ess}(V)  \colon u\in \partial \Sigma_{\mu}\Big\} >0, \quad  \forall
\rho\in \left[\frac{1}{2},1\right].
\end{equation}
We now turn to searching appropriate functions $u_0$ and $u_1$. For $u_0 \in H^1_{\mu,\rad}(\R^N)$ we shall take
\begin{equation*}
u_0 := \frac{\varphi_0}{||\varphi_0||_2} \, \mu 
\end{equation*}
where $\varphi_0\geq 0$ is a minimizer associated to $\lambda_1(V)$ (whose existence is insured by condition $(V2)$). Clearly  
\begin{equation}\label{ajoutx}
E_{\rho}(u_0) - \frac{1}{2}\mu^2\lambda_1(V) = -\frac{1}{q}\|u_0\|_q^q<0, \quad  \forall
\rho\in \left[\frac{1}{2},1\right].
\end{equation}
Let us now prove that $u_0 \in \Sigma_{\mu}$. Setting 
$t_0= \displaystyle \frac{t_{\mu}}{\|\nabla u_0\|_2}$ if we show that $t_0 >1$ this will implies that  $u_0 \in \Sigma_{\mu}$. One the one hand 
$\|\nabla [t_0 u_0]\|_2=t_{\mu}$ and 
\begin{equation}\label{point1}
E_{\mu}(t_0 u_0) - \frac{1}{2} (t_0 \mu)^2 \lambda_{ess}(V) <0, \quad  \forall
\rho\in \left[\frac{1}{2},1\right].
\end{equation}
On the other hand, taking into account \eqref{L1501}, we have
\begin{equation}\label{L1502}
\inf\Big\{E_{\rho}(u) - \frac{1}{2}  \mu^2 \lambda_{ess}(V)  \colon u\in H^1_{\rad}(\R^N),\ \|u\|_2 <\mu, \ \|\nabla u\|_2= t_{\mu}\Big\} >0.
\end{equation}
Combining \eqref{point1} and \eqref{L1502} we deduce that $t_0 \mu > \mu$, namely that $t_0>1.$ This proves that $u_0 \in \Sigma_{\mu}$.

Now, for an arbitrary non negative function $\varphi \in H^1_{\mu,\rad}(\R^N) \backslash \{0\}$ we define, for $t >0$, the function $\varphi_t(x):=t^\frac{N}{2}\varphi(tx)$. Observe that, for all $t>0$, $\|\varphi_{t}\|_2=\mu$ and
\begin{align}
E_{\rho}(\varphi_t)&=\frac{1}{2}\int_{\R^N}|\nabla \varphi_t|^2+\frac12\int_{\R^N}V(x)|\varphi_t^2|-\frac{\rho}{q}\int_{\R^N}|\varphi_t|^q \nonumber\\
&\leq  \frac{t^2}{2}\int_{\R^N}|\nabla \varphi|^2+\frac12 ||V||_{\infty} \, \mu^2-\frac{\rho}{q}t^\frac{N(q-2)}{2}\int_{\R^N}|\varphi|^q. \label{eq:MP endpoints}
\end{align}
Noting that $\frac{N(q-2)}{2} >2$ we see, from \eqref{eq:MP endpoints}, that $E_{\rho}(\varphi_t)\to -\infty$ as $t\to +\infty$. Thus we may take $t_1 >0$ sufficiently large so that $u_1:=\varphi_{t_1} \notin \overline\Sigma_{\mu}$ and
\begin{equation}\label{eq:MP endpoints2}
    E_{\rho}({u_1})\leq E_{\rho}(u_0), \qquad \forall
\rho\in \left[\frac{1}{2},1\right].
\end{equation} 
In view of the definition of $\Sigma_{\mu}$ in \eqref{defSigma}, of \eqref{L_positive infimum}, \eqref{ajoutx} and \eqref{eq:MP endpoints2}, the lemma is now proved.
\end{proof}

\begin{remark}\label{ALR}
From the proof of Lemma \ref{lemma: the mountain pass} we directly deduce,  under assumptions $(V0)-(V2)$, that for all $\mu \in (0,\mu_0)$,
\begin{equation*}
\delta_{\mu}:= \inf \Big\{E_{\rho}(u)\colon u\in \Sigma_{\mu}\Big\}< \frac{1}{2}\mu^2 \lambda_1(V) < \inf_{u\in\partial\Sigma_{\mu}}E_{\rho}(u)
\end{equation*}
where $\Sigma_{\mu}$ is defined in \eqref{defSigma}. In particular $E_{\rho}$ has a geometry of local minimum on $H^1_{\mu, \rad}(\R^N)$ and any minimizing sequence for $E_{\rho}$ inside $\Sigma_{\mu}$ is bounded and stays away from the boundary $\partial\Sigma_{\mu}$.
\end{remark}

We now establish a series of results which will be useful in deriving information on the Lagrange multipliers associated to the limit function of the Palais-Smale sequences provided by the application of Proposition \ref{lemma:result critical point theory}. 

\begin{lemma}\label{louis2}
Assume that $(V0)-(V1)$ hold. For every $\lambda< - \lambda_{\ess}(V)$ and every $u\in H^1_\rad(\R^N)$ there exists a subspace of infinite dimension $Y$ of $H^1_\rad(\R^N)$ such that for all $\varphi\in Y  \backslash \{0\}$
$$\int_{\R^N}|\nabla\varphi|^2+\int_{\R^N}(V(x)+\lambda-(q-1)\rho|u|^{q-2})\varphi^2<0.$$
\end{lemma}
\begin{proof}

Since $\lambda_{\ess}(V) = \inf\sigma_\ess(-\Delta_\rad+V)$, for every $0<\epsilon<-\lambda-\lambda_\ess$ there exists $\lambda_\epsilon\in \sigma_\ess(-\Delta_\rad+V)$ such that
$$\lambda_\epsilon<\lambda_\ess(V)+\epsilon.$$
Fix one such $\epsilon$. Using Weyl's criterion, see \cite[Theorem 7.2]{HiSi}, we obtain the existence of a sequence  $\{\varphi_n\}$ of functions in $H^2_\rad (\R^N)$, such that $\|\varphi_n\|_2=1,$ $\|(-\Delta_\rad+V -\lambda_{\epsilon})\varphi_n\|_2\to 0$ and $\varphi_n\rightharpoonup 0$ in $L^2(\R^N)$, otherwise known as a Weyl sequence. Note that the last property implies that the subspace generated by the sequence is infinitely dimensional. Moreover, from the first and second properties we deduce, 
\begin{align*}
       \int_{\R^N}|\nabla\varphi_n|^2+\int_{\R^N}(V(x) + \lambda &- (q-1)\rho|u|^{q-2}) |\varphi_n^2|  \\&\leq \int_{\R^N}|\nabla\varphi_n|^2+\int_{\R^N}(V(x) + \lambda)|\varphi_n^2| \\&= \int_{\R^N}|\nabla\varphi_n|^2+\int_{\R^N}(V(x) - \lambda_{\epsilon})|\varphi_n^2| + (\lambda + \lambda_{\epsilon}) \int_{\R^N}|\varphi_n^2|  \\& = o_n(1) + (\lambda + \lambda_{\epsilon}).
    \end{align*}
 The claim is now proven by taking $Y=\{\varphi_n\}_{n\geq N}$ for some $N \in \N$ sufficiently large.
\end{proof}

\begin{lemma}\label{lemma: weak potential conditions}
    Assume that $(V0)-(V2)$ holds. Suppose that $u$ is a positive solution to \begin{equation}\label{eq:general schro}
    -\Delta u+V(x)u+\lambda u=|u|^{q-2}u
\end{equation} for some $\lambda\in \R.$ Then 
		$\lambda > - \lambda_1(V) = \inf \sigma(-\Delta+V)$.
\end{lemma}
\begin{proof}
     Consider on $H^2_\rad (\R^N)$ the Schr\"odinger operator $Qv:=-\Delta v+Hv$  
with $H:=V-|u|^{q-2}$; we claim that 
     \begin{equation}\label{eq:characterisation lagrange parameter}
         \min \sigma(Q)=-\lambda.
     \end{equation}
Indeed, since $u$ is a solution to \eqref{eq:general schro} there holds $Qu=-\lambda u$ and thus trivially $\min\sigma(Q)\leq -\lambda$. To now prove that $\min \sigma(Q)\geq -\lambda,$ suppose for contradiction that $\min \sigma(Q)<-\lambda.$ Since then
    $$-\lambda>\min \sigma(Q)=\inf_{\varphi\in C_c^\infty\setminus\{0\}}\frac{\int |\nabla \varphi|^2+ \int H |\varphi^2|}{\int|\varphi|^2},$$ there exists a radial function
		$\phi\in C_c^\infty(\R^N) $ such that
    \begin{equation}
		\label{eq:test spectrum}
        -\lambda > \frac{\int |\nabla \phi|^2+ \int H |\phi^2|}{\int|\phi|^2}.
    \end{equation}
    Now, let $B$ be an open ball in $\R^N$ such that $\supp(\phi)\subset B$, let $\lambda_1$ be the first Dirichlet eigenvalue of $(-\Delta + H)$ on $B$ and let $e_1$ be its associated positive eigenfunction. From the definition of $\lambda_1$ and \eqref{eq:test spectrum} we have that $\lambda_1<-\lambda$. However,
    \begin{align*}
        -\lambda\int_B ue_1&=\int_B(-\Delta u+Hu)e_1=\int_B(\nabla u\cdot\nabla e_1+ Hue_1)=\int_B(-\Delta e_1+He_1)u+\int_{\partial B}\partial_\nu e_1u\\&=\lambda_1 \int_B ue_1+\int_{\partial B}\underbrace{\partial_\nu e_1}_{<0}\underbrace{u}_{>0}<\lambda_1 \int_B \underbrace{ue_1}_{>0}
    \end{align*}
		which results in a contradiction and so \eqref{eq:characterisation lagrange parameter} is proven. 
Now, let $\varphi_0 \in H^1_{\rad}(\R^N) \backslash \{0\}$ be a non-negative minimizer of $\lambda_1(V)$. We have
        $$-\lambda=\min\sigma(Q)\leq \frac{\int |\nabla \varphi_0|^2+ \int H|\varphi_0^2|}{\int|\varphi_0|^2}<\frac{\int |\nabla \varphi_0|^2+ \int V|\varphi_0^2|}{\int|\varphi_0|^2} = \lambda_1(V),$$
finishing the proof.
\end{proof}
We are now in position to give the proof of Theorem \ref{theo main aproximation}. 

\begin{proof}[Proof of Theorem \ref{theo main aproximation}.] In view of Lemma \ref{lemma: the mountain pass} we can apply Proposition \ref{lemma:result critical point theory} to obtain, for almost every $\rho\in[\frac12,1]$, a bounded Palais Smale sequence $\{u_n\}$ in $H^1_{\mu,\rad}(\R^N)$ for the functional $E_{\rho}$ constrained to $H_{\mu,\rad}^1(\R^N)$.
Namely,
\begin{equation*}
E_{\rho}(u_n) \to c_{\rho}
\end{equation*}
and, see Remark \ref{remark almost lagrange},
\begin{equation}\label{const crit}
  E'_{\rho}(u_n) +  \lambda_n (u_n,\cdot) \to 0
  \quad \text{in the dual of }  H^1_{\mu,\rad}(\R^N),
\end{equation}
where 
\begin{equation}\label{lambda}
  \lambda_n:= -\frac{1}{\mu} E'_\rho(u_n)[u_n].
\end{equation}
Moreover, from Proposition \ref{lemma:result critical point theory} (iv), there exists a sequence $\{\zeta_n\}\subset \mathbb{R}^+$ with $\zeta_n\to 0^+$ such that, if the inequality
\begin{align} \label{L-Hess crit}
\int_{\R^N}|\nabla\varphi|^2+(V(x)+\lambda_n-(q-1)\rho&|u_n|^{q-2})|\varphi^2|dx
  < -\zeta_n \|\varphi\|_{H^1(\R^N)}^2
\end{align}
holds for any $\varphi \in W_n \setminus \{0\}$ for $W_n$ a subspace of
$T_{u_n} H^1_{\mu,\rad}(\R^N)$, then the dimension of $W_n$ is at most
$1$. The above inequality uses the following relation
\begin{equation}\label{eq second derivative}
    \int_{\R^N}|\nabla\varphi|^2+(V(x)+\lambda_n-(q-1)\rho|u_n|^{q-2})|\varphi^2|dx
  = E''_\rho(u_n)[\varphi, \varphi]
  + \lambda_n \|\varphi\|_{L^2(\R^N)}^2,
\end{equation}
which is obtained from direct computation. In addition, as explained in \cite[Remark 1.8]{BoChJeSo24}, since $u_0,$ $u_1 \geq 0$, $u \in H_{\mu,\rad}^1(\R^N) $ implies that $ |u| \in H_{\mu,\rad}^1(\R^N)$,  the map $u \mapsto |u|$ is continuous and $E_{\rho}(u) = E_{\rho}(|u|)$, it is possible to choose $\{u_n\}$ with the property that $u_n \geq 0$, for all $n \in \N$.
 
Since $\{u_n\} \subset {H^1_{\mu,\rad}(\R^N)}$ is also bounded, using \eqref{lambda} it follows that the sequence of almost Lagrange multipliers $\{\lambda_n\} \subset \R$ is bounded. Therefore, passing to subsequences, there exist $\lambda_\rho\in \mathbb{R}$ and
$u_\rho\in H^1_\rad(\R^N)$ such that
\begin{align}\lim_{n\to\infty}&\lambda_n=\lambda_{\rho},\label{3-5-1}\\
  u_n\rightharpoonup\ &u_{\rho}
  \quad\text{in } H^1_{\rad}(\R^N),
    \label{3-5-2}\\
    u_n\to\ &u_{\rho} 
  \quad\text{in } L^s_\rad(\R^N), \quad 2<s<2^*,
    \label{3-5-3}\\
     u_n(x)\to\ &u_\rho(x) \quad \text{a.e. in }\ x\in \R^N \label{3-5-4}.
\end{align}
For any  $\eta\in H^1_\rad(\R^N)$, we have from \eqref{const crit} that
\begin{align}
  0
  &= \lim_{n\to\infty} \bigl( E_{\rho}'(u_n)
    + \lambda_n (u_n, \cdot) \bigr) [\eta] \nonumber\\ 
  &= \lim_{n\to\infty}\left[
    \int_{\R^N} \nabla u_n \nabla\eta  + \int_{\R^N} V(x)u_n\eta + \lambda_n \int_{\R^N} u_n\eta
    - \rho \int_{\R^N}|u_n|^{q-2}u_n \eta\right].\label{Eq:limit PS condition2}
\end{align}
Note as well that, since $\{u_n\} \subset {H^1_{\mu,\rad}(\R^N)}$ is bounded, from \eqref{const crit}, \eqref{3-5-1} we have that
\begin{equation}\label{AL1}
\int_{\R^N} |\nabla u_n|^2 + V(x) |u_n|^2  + \lambda_{\rho}\int_{\R^N}|u_n|^2 - \int_{\R^N} \rho |u_n|^q= o_n(1), \quad \mbox{as } n \to \infty.
\end{equation}
From \eqref{Eq:limit PS condition2} and using \eqref{3-5-1}--\eqref{3-5-3} we deduce that
\begin{equation}\label{Eq:weakurho}
0 = \int_{\R^N} \nabla u_\rho \nabla\eta + \int_{\R^N} V(x)u_\rho\eta+\lambda_\rho \int_{\R^N} u_\rho\eta
    - \rho\int_{\R^N}|u_\rho|^{q-2}u_\rho\eta.
\end{equation}
This shows that $u_\rho$ is a weak solution to
$$ - \Delta u + V(x) u + \lambda_{\rho}u = \rho|u|^{q-1}u, \quad u \in H^1_{rad}(\R^N).$$
Let us now  show the strong convergence of $\{u_n\} \subset {H^1_{\mu,\rad}(\R^N)}$.  First, since the codimension of $T_{u_n} H^1_{\mu,\rad}(\R^N)$ in $H^1_\rad(\R^N)$ is $1$, we infer that if inequality \eqref{L-Hess crit} holds for every $\varphi \in W_n \backslash \{0\}$ in a subspace $W_n$ of $H^1_\rad(\R^N)$ of dimension at most 1, then the dimension of $W_n$ is at most 2. Second, by Lemma \ref{louis2}, for every $\lambda < - \lambda_{ess}(V)$, there exists a subspace $Y$ of $H^1_\rad(\R^N)$ with $\dim Y \geq 3$ and $a>0$ such that, for $n \in \N$ large,
\begin{equation*}
E_\rho''(u_\rho)[w,w]+\lambda_\rho[w,w] \leq \int_{\R^N} |\nabla u_\rho|^2 + V(x)|u_\rho^2| + \lambda \int_{\R^N}|w|^2 \leq - a ||w||^2, \quad \mbox{for all } w \in Y \backslash \{0\}.
\end{equation*}
Therefore, combining Lemma \ref{lemma:Louis1} with the two previous observations we obtain that 
\begin{equation}\label{eq lambda non negative}
    \lambda_{\rho}\geq - \lambda_{ess}(V).
\end{equation}
Now, choosing $\eta = u_{\rho}$ in \eqref{Eq:weakurho}
and subtracting \eqref{Eq:weakurho} from \eqref{AL1}
we get, using \eqref{3-5-2}--\eqref{3-5-3},
\begin{equation}\label{from weak to strong convergence}
    \int_{\R^N} |\nabla u_n-\nabla u_\rho|^2 +\int_{\R^N}  V(x)|u_n-u_\rho|^2 + \lambda_\rho \int_{\R^N} |u_n-u_\rho|^2\to 0.
\end{equation} 
From \eqref{AL1} we may deduce that $u_\rho$ is non-zero. Indeed assuming by contradiction that $u_\rho\equiv 0$ 
and thus that $u_n \to 0$ in $L^q(\R^N)$, we have, in view of \eqref{eq:MP definition},
\begin{align}
    \frac{1}{2}\mu^2 \lambda_{ess}(V) < c_\rho =&\lim_{n\to\infty} E_\rho(u_n)= \lim_{n\to\infty}
    \frac{1}{2}\int_{\R^N}|\nabla u_n|^2+ \frac{1}{2} \int_{\R^N}V(x)|u_n|^2-\frac{\rho}{q}\int_{\R^N}|u_n|^q
    \nonumber\\ 
    =&\lim_{n\to\infty}\frac{1}{2}\left[ \int_{\R^N}|\nabla u_n|^2 + \int_{\R^N}V(x)|u_n^2|\right]+o_n(1) 
\end{align}
from which we deduce using  \eqref{AL1} that
\begin{equation}\label{negative} 
\lambda_{\rho}< - \lambda_{ess}(V).
\end{equation}
Combining \eqref{eq lambda non negative} and \eqref{negative} we obtain a contradiction. This proves that $u_{\rho} \not \equiv 0.$ Now, in view of \eqref{3-5-4} we obtain that $0\not\equiv u_{\rho}\geq 0$ and so by the strong maximum principle, see for example \cite[Proposition 1.3]{LaMo24}, $u_\rho>0$. In particular, we may now apply Lemma \ref{lemma: weak potential conditions} allowing to improve \eqref{eq lambda non negative} to
\begin{equation}\label{eq lambda positive}
    \lambda_{\rho}> -\lambda_1(V).
\end{equation}
Knowing that \eqref{eq lambda positive} holds, we directly deduce from \eqref{from weak to strong convergence} that $u_n\to u_\rho$ in $H^1_\rad(\R^N)$.

Lastly, it remains to show that the Morse index of $u_\rho$ is at most 2. 
Observe that from Remark \ref{strong-convergence} we already know that the Morse index of $u_\rho$ as a constrained critical point is at most 1. 
Now, from \eqref{eq second derivative}, for every $w\in H^1(\R^N)$ there holds that
\begin{align*}
    D^2E_\rho(u_\rho)[w,w]:&=E_\rho''(u)[w,w]+\lambda_\rho[w,w]
    =Q^N_V(w;u_\rho).
\end{align*}
Thus, using in a similar manner as before that $H_{\mu,\rad}^1(\R^N)$ is of codimension 1 in $H^1_\rad(\R^N)$ along with Remark \ref{rem: morse connexions} we directly deduce
\begin{equation}\label{eq:Morse index urho}
    \morse(u_\rho) \leq 2,
\end{equation}
completing the proof.
\end{proof}

Let us now show how Theorem \ref{theo:main} can be proven using Theorem \ref{theo main aproximation} and assuming Theorem \ref{theo maxima to 0}.
The proof of the latter will be given in Section \ref{section: blow up}.
\begin{proof}[Proof of Theorem \ref{theo:main}.]
    Theorem \ref{theo main aproximation} gives the existence, for almost every $\rho\in[\frac12,1]$, of a solution $(u_\rho,\lambda_\rho)\in H^1_{\mu,\rad}{(\R^N)\times \R}$ to \eqref{EQ:bounded approximation}. In particular $\|u_\rho\|_{L^2(\R^N)}=\mu$ and
		\begin{equation}\label{ALL1}
E_{\rho}(u_{\rho})=c_{\rho}=\inf_{\gamma\in \Gamma}\max_{t\in [0,1] }E_{\rho}(\gamma(t))\leq \inf_{\gamma\in \Gamma}\max_{t\in [0,1] }E_{\frac12}(\gamma(t))=c_\frac12,
\end{equation}
showing that their energy is uniformly bounded. In addition, recall that $\morse(u_{\rho}) \leq 2$. {Choose a sequence $\{\rho_n\} \subset [\frac12,1]$ with $\rho_n \to 1^-$ whose associated sequence of solutions is $\{(u_{\rho_n},\lambda_{\rho_n})\} \subset H^1_{\mu,\rad}(\R^N){\times \R}$,} which we will denote by 
$\{(u_n,\lambda_n)\}$ for simplicity.
One finds that the triplets $\{(\rho_n,\lambda_n,u_n)\}\in[\frac12,1]\times\R\times H_\rad^1(\R^N)$ 
are {solutions} to \eqref{EQ:bounded approximation} such that \eqref{eq mass energy unif bound} holds and $\sup_{n\in \N} \morse(u_n) \leq 2$. Applying Theorem \ref{theo maxima to 0} we deduce that the sequence $\{\lambda_n\} \subset \R$ is bounded. Now, observing that
\begin{align*}
    E_{\rho_n}(u_{n})&=\frac{1}{2}\int_{\R^N}|\nabla u_{n}|^2+\frac{1}{2} \int_{\R^N}V(x)|u_{n}|^2-\frac{\rho_n}{q}\int_{\R^N}|u_{n}|^q \nonumber
    \\
    &=\left(\frac{1}{2}-\frac{1}{q}\right)\int_{\R^N}|\nabla u_{n}|^2+\left(\frac{1}{2}-\frac{1}{q}\right)\int_{\R^N}V(x)|u_{n}|^2-\frac{\lambda_{n}}{q}\int_{\R^N}|u_{n}|^2 \nonumber
    \\
    &\geq \left(\frac{1}{2}-\frac{1}{q}\right)\int_{\R^N}|\nabla u_{n}|^2 +\left(\frac{1}{2}-\frac{1}{q}\right)\underline{V}\int_{\R^N}|u_{n}^2|-\frac{\lambda_{n}}{q}\int_{\R^N}|u_{n}|^2,
\end{align*}
it follows that
\begin{equation}\label{eq:gradient boundedness Rn}
    \left(\frac{q-2}{2q}\right)\|\nabla u_n\|_2^2\leq E_{\rho_n}(u_n)+\frac{\lambda_n\mu^2}{q} - \left(\frac{q-2}{2q}\right)\underline{V}\mu^2.
\end{equation}
Thus, in view of \eqref{ALL1}, knowing that $\{\lambda_n\} \subset \R$ is bounded implies that $\{u_{n}\}$ is bounded in $H^1(\R^N)$. 
Also, as in \cite[Lemma 2.3]{Je99}, we can prove that $\rho \mapsto c_{\rho}$ is a continuous function from the left and thus that
$$\lim_{\rho_n \to 1^-}c_{\rho_n}=c_1.$$
At this point, it is easily seen that $\{u_{n}\}$ is a (bounded) PS sequence for $E = E_1$ constrained to $H^1_{\mu,\rad}(\R^N)$ at the level $c_1$. Then, proceeding as in the proof of Theorem \ref{theo main aproximation}, we can show that $\{u_n\}$ converges to a limit $u\in H^1_\mu(\R^N)$ which, by the previous argument, is a critical point of $E$ constrained to $H^1_{\mu,\rad}(\R^N)$ at mountain pass level $c_1$. In all, $u$ is the sought solution that satisfies the conclusions of Theorem \ref{theo:main}. 
\end{proof}

We end this section by proving the existence of a second solution as a constrained local minimum. Here we directly work with the Energy functional $E$ defined in \eqref{eq:functional Rn}. 
\begin{proof}[Proof of Theorem \ref{theo: local min}.]
In Remark \ref{ALR} we have observed that, for all $0 < \mu < \mu_0$, there exists a bounded {set $\Sigma_{\mu}$} such that
\begin{equation}\label{eq negative local min}
\delta_{\mu}:= \inf \Big\{E(u)\colon u\in \Sigma_{\mu}\Big\}< \frac{1}{2}\mu^2 \lambda_1(V) < \inf_{u\in\partial\Sigma_{\mu}}E(u)
\end{equation}
where $E=E_1$. Let $\{u_n\} \subset \Sigma_{\mu}$ be a minimizing sequence for $\delta_{\mu}$ such that every $u_n$ is a non negative function. Invoking Ekeland's variational principle, without loss of generality we may assume that there is an arbitrarily close minimizing sequence which is in addition a Palais-Smale sequence. Since, by construction, the sequence is bounded, passing to a subsequence we have $u_n\rightharpoonup u \geq 0$ in $H^1_\rad(\R^N)$ and $u_n \to u$ in $L^q_\rad(\R^N)$. Due to the first inequality in \eqref{eq negative local min} there necessarily holds $u \not \equiv 0$ and thus $u>0$ by the strong maximum principle. At this point we can proceed as in the proof of Theorem  \ref{theo main aproximation} to conclude 
 that $u_n\to u$ in $H^1_\rad(\R^N)$ with $u \in H^1_{\mu}(\R^N)$ a constrained critical point of $E$ at the level $\delta_{\mu}$. Now, combining \eqref{eq:MP definition} used with $\rho =1$ and \eqref{eq negative local min}, we deduce that
$$c_1 \geq \inf_{u\in\partial\Sigma_{\mu}}E(u) > \frac{1}{2}\mu^2 \lambda_1(V) > \delta_{\mu}.$$
This guarantees that the Energy level of the local minimizer is strictly below the one of the solution found in Theorem \ref{theo:main} and ends the proof.

\end{proof}

\section{A blow-up analysis and proof of Theorem \ref{theo maxima to 0} }\label{section: blow up}

The aim of the section is to prove Theorem \ref{theo maxima to 0}.

We begin by showing that if $\lambda_n\to \infty$ then sequences of local maxima actually blow-up. 
\begin{lemma}\label{Lemma maxima bounds}
    Let $(V1)$ hold and $\{(\rho_n,\lambda_n,u_n)\}\in[\frac12,1]\times\R\times H_\rad^1(\R^N)$ be a sequence of solutions to Equation \eqref{L_EQ:bounded approximation}.
  Suppose $x_n$ is a local maximum of $u_n$. Then
    $$ (\lambda_n + \underline{V})^\frac{1}{q-2}\leq  u_n(x_n).$$
\end{lemma}
\begin{proof}
    Since $(V1)$ holds recall that $u_n\in W^{2,\infty}_\loc(\R^N)$, for any $n \in \N$. We have $-\Delta u_n(x_n)\geq 0$ and so 
    \begin{align*}
        0\leq - \Delta u_n(x_n) = \rho_n(u_n(x_n))^{q-1}- (V(x_n)+ \lambda_n) u_n(x_n).
    \end{align*}
    Dividing by $u_n(x_n)>0$, rearranging and computing we get
\begin{align*}
    (\lambda_n + \underline{V})\leq V(x_n)+\lambda_n&\leq (u_n(x_n))^{q-2}
\end{align*}
and the claim follows.
\end{proof}

The next two lemmas are related to certain limiting problems of interest.
\begin{lemma}\label{lemma: morse limit problem}
Let $U$ be a non trivial non-negative solution to
\begin{equation}\label{CaseN}
- \Delta u+\lambda u= u^{q-1}\quad x	\in \R^d, \quad d \geq 2
\end{equation}
where $\lambda >0$, having a finite Morse index. Then
\begin{itemize}
\item[i)] $U$ decays to $0$ at infinity, $U \in H^1(\R^d)$ and it is the unique positive solution to \eqref{CaseN}. 
\item[ii)] The Morse index of $U$, as defined in Definition \ref{def Morse index}, is positive.
\end{itemize}
\end{lemma}

\begin{proof}
These results when $\lambda=1$ are given in \cite[Theorem 1.1]{EsPe11}. As it can be checked following its proof, see in particular \cite[Remark 2.2 and Theorem 2.3]{EsPe11}, they extend to any $\lambda >0$.
\end{proof}

\begin{proposition}\label{lemma lagrange phase plane}
Let $U\in C^2(\R)$ be the solution to the ODE
\begin{equation}\label{equationdepart}
    \begin{cases}
         &  -U''+\beta U= U^{q-1}, \quad \mbox{in } \R, \\
         &  0 < U(r) \leq U(0) =1, \quad \mbox{in } \R,
    \end{cases}
		\end{equation}
		for some $\beta\geq 0$. Then $\beta\geq \frac{2}{q}$ and
    \begin{itemize}
        \item if $\beta>\frac{2}{q}$, then the Morse index of $U$ is infinite;
        \item if $\beta=\frac{2}{q},$ then $U$ belongs {to $H^1(\R)$,} decays exponentially and
        \begin{equation*}
\frac{2}{q}\int_\R U^2=\left(\frac12+\frac{1}{q}\right)\int_\R U^q.
        \end{equation*}
				\end{itemize}
\end{proposition}

\begin{proof}
When $\beta >0$, the proposition is proved in \cite[Appendix A.1]{EMSS}. When $\beta =0$, a phase plane analysis of the equation $-u'' = |u|^{q-2}u$ in $\R$ shows that its non-trivial solutions are periodic sign-changing.
\end{proof}

From the above proposition we directly deduce 
\begin{lemma}\label{2_lemma lagrange phase plane}
Let  $U\in C^2(\R)$ be the solution to the ODE

\begin{equation}\label{equationarrivee}
    \begin{cases}
         &  -U''+  U= U^{q-1}, \quad \mbox{in } \R, \\
         &  0 < U(r) \leq U(0) = \left(\frac2{q}\right)^{-\frac1{q-2}}, \quad \mbox{in } \R.
    \end{cases}
		\end{equation}
    Then $U\in H^1(\R)$, decays exponentially and 
    \begin{equation}\label{norm relation sphere}
    \frac{2q}{q+2}\int_\R U^2 = \int_\R U^q.
    \end{equation}
	In addition $U$ has a positive Morse index.  
\end{lemma}
\begin{proof}
The lemma follows from Proposition \ref{lemma lagrange phase plane} observing that  $V$ is solution to \eqref{equationdepart} with $\beta = \frac{2}{q}$, if and only if,
$$U(r) = \Big(\frac{2}{q} \Big)^{- \frac{1}{q-2}} V {\Big(\Big(\frac{2}{q}\Big)^{- \frac{1}{2}}r \Big)}$$
is solution to \eqref{equationarrivee}. We have
{
\begin{align*}
    Q_{0}^1(U;U)&= \int_{\R} |{\nabla U}|^2 + |U|^2 - (q-1)|U|^{q}
    =(2-q) \int_{\R}|U|^q<0,
\end{align*}
}
and thus the Morse index of $U$ is positive.
\end{proof}
We are now ready to perform the blow-up analysis with the following two theorems. In Theorem \ref{mainblow} we study the qualitative behavior of solutions to Equation \eqref{L_EQ:bounded approximation} as $\lambda\to +\infty$ at close neighbourhoods of some points of maxima, which, by Lemma \ref{Lemma maxima bounds} we know must blow-up. In what follows we will often work with the unidimensional associated functions $u(r)$ of a radial function $u(x)$ defined in $\R^N$. 
Additionally, we also consider the unidimensional derivatives of $u(r)$ denoted by $u'(r)$ and $u''(r).$ 

\begin{theorem}\label{mainblow}
Let $(V0)-(V1)$ hold and $\{(\rho_n,\lambda_n,u_n)\}\in[\frac12,1]\times\R\times H_\rad^1(\R^N)$ be a sequence of solutions to \eqref{L_EQ:bounded approximation}. Suppose that $\rho_n\to 1^-$ and $\lambda_n\to+\infty$. Suppose moreover that the Morse index of the sequence $\{u_n\}$ is uniformly bounded, i.e. there exists $m^*<+\infty$ such that,
$$
\sup_{n\in \N} \morse(u_n) \leq m^*.
$$
Let $\{a_n\} \subset [0, + \infty)$ be such that, for some sequence $R_n\to +\infty$,
$$ u_n(a_n)=\max\{u_n(r) \colon\ |r-a_n|<R_n\tilde\epsilon_n\}\quad \text{with} \quad \tilde\epsilon_n:=u_n(a_n)^{-\frac{q-2}{2}}\to 0.$$
Suppose moreover that 
\begin{equation}\label{eq: blow up on sphere}
    \limsup_ {n\to\infty}\frac{a_n}{\tilde\epsilon_n}=+\infty.
\end{equation}
Then, up to subsequences,
\begin{itemize}
  \item[(a)] setting $\epsilon_n:=\lambda_n^{-\frac{1}{2}}$  there holds
\begin{equation}\label{eq:eps n ratio sphere}
    \lim_{n\to\infty}\frac{{\tilde\epsilon_n}}{\epsilon_n}
    =\sqrt{\frac2 q}
\end{equation}
and 
\begin{equation}\label{eq: maxima ration shpere}
    \lim_{n\to\infty}\frac{a_n}{\epsilon_n}=+\infty;
\end{equation}
\item[(b)] the one dimensional scaled sequence
\begin{equation}
\label{eq:good rescaling sphere}
    v_n(s)=\lambda_n^{-\frac{1}{q-2}}u_n\left(\frac{s}{\sqrt{\lambda_n}}+a_n\right), \quad s\in \left(-\lambda_n^{\frac12}{R_n}{\tilde\epsilon_n},\lambda_n^{\frac12}{R_n}{\tilde\epsilon_n}\right),
\end{equation}
converges in $C_{\mathrm{loc}}^1(\R)$ to the (unique) positive solution $\phi \in H^1(\R)$ of
    \begin{equation}
		\label{Eq scaled convergence sphere}
		\left\{\begin{array}{lll}
    -\phi'' +  \phi=  \phi^{q-1}, \quad \text{in}\quad \R \\
    \phi >0\\
    \max\phi=\phi(0)=\left(\frac2{q}\right)^{-\frac{1}{q-2}}
    ;
    \end{array}
        \right.
    \end{equation}
\item[(c)] for each $n \in \N$, there exists a radial function $\phi_n\in C_c^\infty(\R^N)$ with support contained in the annulus
$$A(a_n,\bar R\epsilon_n)=\{x\in\R^N\colon 0<a_n- \bar R\epsilon_n\leq |x|\leq a_n+\bar R\epsilon_n\}$$
 for some $\bar R>0$ such that 
\begin{equation}\label{eq: morse sequence sphere}
    Q_V^N(\phi_n,u_n)<0.
\end{equation}

\end{itemize}
Now consider the complementary case, i.e. suppose that 
\begin{equation}\label{eq: blow up origin}
    \limsup_ {n\to\infty}\frac{a_n}{\tilde\epsilon_n}<+\infty.
\end{equation}
Then, passing to subsequences if necessary,
\begin{itemize}
  \item[(a')] setting $\epsilon_n:=\lambda_n^{-\frac{1}{2}}$  there holds
	
	\begin{equation}\label{eq:eps n ratio origin}
    \sqrt{\tilde\lambda}:=\lim_{n\to\infty}\frac{{\tilde\epsilon_n}}{\epsilon_n}
    \in (0,1]
\end{equation}
and 
\begin{equation}\label{eq: good ratio origin}
    \lim_{n\to\infty}\frac{a_n}{\epsilon_n}=0;
\end{equation}

\item[(b')] the scaled sequence,
\begin{equation}
\label{eq:good rescaling origin}
    v_n(y)=\lambda_n^{-\frac{1}{q-2}}u_n\left(\frac{y}{\sqrt{\lambda_n}}\right), \quad y\in B\left(0,\lambda_n^{\frac12}{R_n}{\tilde\epsilon_n}\right),
\end{equation}
converges in $C_{\mathrm{loc}}^1(\R^N)$ to the unique positive radial solution $\phi \in H^1(\R^N)$ to
\begin{equation}\label{Eq scaled convergence origin}
-\Delta\phi +  \phi=  \phi^{q-1}, 
\end{equation}

\item[(c')] for each $n \in \N$, there exists a radial function $\phi_n\in C_c^\infty(\R^N)$ with support contained in the ball
$$B(0, a_n + \bar R \epsilon_n) =  \{x\in\R^N\colon  |x|\leq a_n+\bar R\epsilon_n\}$$

 for some $\bar R>0$ such that 
\begin{equation}\label{eq:morse sequence origin}
    Q_V^N(\phi_n,u_n)<0;
\end{equation}
\item[(d')] for all $R>0$ and $s\geq 2$ we have
  \begin{equation}
    \label{eq:mass-convergence origin}
    \lim_{n \to \infty} \lambda_n^{\frac{N}{2}-\frac{s}{q-2}}\int_{B(0, a_n + R\lambda_n^{-1/2})}|u_n(x)|^sdx
=\int_{B(0,R)}|\phi|^s dy\ \in (0,+\infty)
  \end{equation}
\end{itemize}
\end{theorem}

\begin{proof}[Proof of Theorem \ref{mainblow}.]
The radial profiles of the solutions to Equation \eqref{L_EQ:bounded approximation} 
satisfy
\begin{equation*}
    \begin{cases}
         -u''(r)-\frac{N-1}{r}u'(r)+V(r)u(r)+\lambda u(r)=\rho u(r)^{q-1}, \quad &r \in (0, + \infty), \\
        u>0, \quad &r \in [0, + \infty), \\
				u'(0)=0.
    \end{cases}
\end{equation*}
We now treat the two possible cases separately.

{\it Proof under assumption \eqref{eq: blow up on sphere}.} To begin with, suppose that there holds $$\limsup_ {n\to\infty}\frac{a_n}{\tilde\epsilon_n}=+\infty.$$ Note that this includes, but is not restricted to, the case where up to subsequences $a_n\not\to 0$. 
Without loss of generality we may assume that 
\begin{equation*}
    R_n=o\left(\frac{a_n}{\tilde\epsilon_n}\right).
\end{equation*}
Consider then the rescaling
\begin{equation*}
        \tilde u_n(r)=\tilde\epsilon_n^{\frac{2}{q-2}}u_n(\tilde\epsilon_nr+a_n), \quad r\in (-R_n,R_n).
    \end{equation*}
  Observe that the functions $\tilde u_n(r)$ attain their maximum at $r=0$. A straightforward calculation then shows that each $\tilde u_n(r)$ satisfies
\begin{equation*}
    \begin{cases}
         \quad -\tilde u_n''(r)-(N-1)\frac{\tilde\epsilon_n}{\tilde\epsilon_nr+a_n}
         \tilde u_n'(r)+V(r)\tilde\epsilon^2_n\tilde u_n(r)+\lambda_n\tilde\epsilon_n^2 \tilde u_n(r)=\rho_n\tilde u_n(r)^{q-1}, &r\in(-R_n,R_n),\\
        \quad \tilde u_n(r)>0, & r\in(-R_n,R_n),\\
        \quad \tilde u_n(0)=1, &\\
        \quad \tilde u_n'(0)=0.
    \end{cases}
\end{equation*}
Note that the sequences $\left( \frac{\tilde\epsilon_n}{\tilde\epsilon_nr+a_n} \right)$ and  $\left(V(r)\tilde\epsilon^2_n \right)$ both tend to $0$ in $\R$ and $L^\infty(\R)$ respectively and also that
$$ 0 \leq \tilde\Lambda :=\lim_{n\to\infty}\lambda_n\tilde\epsilon_n^2\leq 1.$$
Indeed, thanks to Lemma \ref{Lemma maxima bounds}, we have that 
$$0 < \tilde\epsilon_n^2 \lambda_n = {u_n(a_n)}^{-(q-2)} \lambda_n \leq \frac{\lambda_n}{\lambda_n + \underline{V}}\to 1.$$ 
Then, using standard elliptic estimates, see \cite{GiTr}, we can deduce, up to a subsequence, that $\tilde u_n\to \tilde u$ in $C^1_\loc(\R)$, where $\tilde u$ is a non-negative solution to
\begin{equation}\label{eq lame emden 1d}
    \left\{\begin{array}{lll}
    -\tilde u'' + \tilde\Lambda\ \tilde u = \tilde u^{q-1}, \quad \text{in}\ \R,
    \\
    \tilde u(0)=1,\\
    \tilde u'(0)=0.
    \end{array}\right.
\end{equation}
Let us now prove that the Morse index of $\tilde u\in H_\loc^1(\R)$, see Definition \ref{def Morse index in one dimension}, is bounded by $m^*$. Suppose by contradiction $\morse(\tilde u) > m^*$. Then there 
  exist linearly independent functions
  $\varphi_1, \dots, \varphi_m \in H^1(\R) \cap C_c^\infty(\R)$
  with $m > m^*$ such that    $Q_{0}^1(\varphi_i; \tilde u)<0$ for every $i \in \{1, \cdots, m\}$. Define the radial $N$-dimensional functions
\[
\varphi_{i,n}(r):= \tilde \eps_n^{\frac{1}{2}} \varphi_i\left(\frac{r-a_n}{\tilde \eps_n}\right).
\]
Note that
$\supp(\varphi_{i,n})\subseteq \{x\in \R^N\colon\ |x|\in [a_n-R_i\tilde\epsilon_n,a_n+R_i \tilde\epsilon_n]\}$
for some $R_i>0$ and thus $\varphi_{i,n}$ is well defined for $n \in \N$ sufficiently large thanks to \eqref{eq: blow up on sphere}. We can also check that for every fixed $n \in \N$ the functions $\varphi_{1,n},\dots,\varphi_{m,n}$ are linearly independent in $H^1_{\rad}(\R^N)$. Moreover, using the uniform convergence on compact sets we have
\begin{align*}
    {\frac{1}{a_n^{N-1}}}&Q_V^N(\varphi_{i,n};u_n)\\
    &={\frac{1}{a_n^{N-1}}}\int_{\R^N} |{\nabla\varphi_{i,n}(x)}|^2 + (V(x)+\lambda_n - \rho_n(q-1)|u_n(x)|^{q-2}
)\varphi^2_{i,n}(x)   dx\\
&={\frac{|S^{N-1|}}{a_n^{N-1}}}\int_{a_n-R_i\tilde\epsilon_n}^{a_n+R_i\tilde\epsilon_n}r^{N-1}\left[|{\varphi'_{i,n}(r)}|^2 + (V(r)+\lambda_n - \rho_n(q-1)|u_n(r)|^{q-2}
)\varphi^2_{i,n}(r)\right]dr\\
&= {|S^{N-1}|}\int_{\supp \varphi_i}\left(\frac{\tilde\epsilon_n s+a_n}{a_n}\right)^{N-1}\left[\left|{\varphi'_{i}\left(s\right)}\right|^2 + (\tilde\epsilon_n^2V(\tilde\epsilon_n s+a_n)+\tilde\epsilon_n^2\lambda_n - \rho_n(q-1)|\tilde u_n(s)|^{q-2}
)\varphi^2_{i}\left(s\right)\right]ds\\
&\to {|S^{N-1}|}\int_{\supp\varphi_i}\left[|{\varphi'_{i}(s)}|^2 + \left(\tilde\Lambda- (q-1)|\tilde u(s)|^{q-2}
\right)\varphi^2_{i}(s)\right]ds\\
&= {|S^{N-1}|}Q_0^1(\varphi_i;\tilde u)<0.
\end{align*}
where $|S^{N-1}|$ is the size of the $N-1$ dimensional unit sphere. {This implies that $Q_V^N(\varphi_{i,n};u_n)<0$.} This contradicts the assumption that $\morse(u_n) \leq m^*$ and thus $\morse(\tilde u) \leq m^*$ holds.
At this point, applying Proposition \ref{lemma lagrange phase plane} we deduce that $$\tilde \Lambda=\frac{2}{q},$$
proving \eqref{eq:eps n ratio sphere}. This, together with \eqref{eq: blow up on sphere} results in \eqref{eq: maxima ration shpere}, proving (a).

 Consider now the alternative re-scaling, which will prove more convenient later
\begin{equation*}
        v_n(s)=\lambda_n^{-\frac{1}{q-2}}u_n\left(\frac{s}{\sqrt{\lambda_n}}+a_n\right), \quad s\in (-\lambda_n^{\frac12}{R_n}{\tilde\epsilon_n},\lambda_n^{\frac12}{R_n}{\tilde\epsilon_n}),
    \end{equation*}
    which satisfies 
    $$v_n(0)=\lambda_n^{-\frac{1}{q-2}}u_n(a_n)=
    \lambda_n^{-\frac{1}{q-2}}\max_{r\in A(a_n,{R_n}{\tilde\epsilon_n)}}u_n(r).$$
		Following the same steps we can prove that, 
up to a subsequence, the sequence $\{v_n\}$ converges in $C_{\mathrm{loc}}^1(\R)$ to a function $\phi>0$ satisfying
    \begin{equation*}
		\left\{\begin{array}{lll}
    -\phi'' +  \phi=  \phi^{q-1}, \quad\quad \text{in}\quad \R    \\
    \phi\leq \phi(0)=\left(\frac2{q}\right)^{-\frac{1}{q-2}}.
  \end{array}
        \right.
    \end{equation*}
Using Lemma \ref{2_lemma lagrange phase plane} we deduce $\phi \in H^1(\R)$ by identification, and thus (b) holds, Now, (c) is also consequence of Lemma \ref{2_lemma lagrange phase plane}. Indeed, let $\varphi\in C_c^\infty(\R)$ be a function such that $Q_0^1(\varphi;\phi)<0$ with which we define the $N$-dimensional radial function 
$$\varphi_{n}(x):=  \eps_n^{\frac{1}{2}} \varphi\left(\frac{|x|-a_n}{ \eps_n}\right).$$
Observe that $\supp (\varphi_n)\subset A(a_n,\bar R \epsilon_n)$ for some fixed $\bar R$ which depends on the support of $\varphi$. Furthermore, denoting one again by $|S^{N-1}|$ the size of the $N-1$ dimensional unit sphere,
\begin{align*}
    {\frac{1}{a_n^{N-1}}}Q_V^N(\varphi_n;u_n)&={\frac{1}{a_n^{N-1}}}
   { \int_{\R^N}\left[|\nabla\varphi_n(x)|^2+(V(x)+\lambda_n-\rho_n(q-1)|u_n(x)|^{q-2})\varphi_n^2(x)\right]dx}\\
&= {|S^{N-1}|}\int_{\R}\left(\frac{\epsilon_ns+a_n}{a_n}\right)^{N-1}\left[\left|{\varphi'\left(s\right)}\right|^2 + (\epsilon_n^2V(\epsilon_ns+a_n)+\epsilon_n^2\lambda_n - \rho_n(q-1)|v_n(s)|^{q-2}
)\varphi^2\left(s\right)\right]ds\\
&\to {|S^{N-1}|}\int_{\R}\left[|{\varphi'(s)}|^2 + \left(1-(q-1)|\phi|^{q-2}
\right)\varphi^2(s)\right]ds\\
&= {|S^{N-1}|}Q_0^1(\varphi;\phi)<0,
\end{align*}
and so the statement follows.

{\it Proof under assumption \eqref{eq: blow up origin}.} Let us now turn to the case $$\eta:=\limsup_{n\to\infty}\frac{a_n}{\tilde \epsilon_n}\in[0,+\infty).$$
Since $R_n\to+\infty$, for $n \in \N$ sufficiently large, we have that 
\begin{align*}
    u_n(a_n)&=\max\{u_n(r) \colon\ \max\{0,a_n-R_n\tilde\epsilon_n\}< r< a_n+R_n\tilde\epsilon_n\}=\max\{u_n(r) \colon\ r< a_n+R_n\tilde\epsilon_n\}\\
    &=\max\{u_n(r)\colon\ r\in B(0,a_n+R_n\tilde\epsilon_n)\}.
\end{align*}
This time we begin by considering the rescaling 
$$\tilde u_n(r)=\tilde\epsilon_n^{\frac{2}{q-2}}u_n(\tilde\epsilon_nr), \quad r\in [0,{a_n+R_n\tilde\epsilon_n}),$$
whose associated equation is
$$\begin{cases}
         \quad -\tilde u_n''(r)-\frac{N-1}{r}\tilde u_n'(r)+V(r)\tilde\epsilon^2_n\tilde u_n(r)+\lambda_n\tilde\epsilon_n^2 \tilde u_n(r)=\rho_n\tilde u_n(r)^{q-1}, &r\in[0,{a_n+R_n\tilde\epsilon_n}),\\
        \quad \tilde u_n(r)>0, & r\in[0,{a_n+R_n\tilde\epsilon_n}), \\
        \quad \tilde u_n'(0)=0.
    \end{cases}$$
For every $a>\eta+1$, and $n \in \N$ sufficiently large, there holds by definition,
$$\tilde u_n\left(\frac{a_n}{\tilde\epsilon_n}\right)=1=\max_{r<a}  \tilde u_n(r).$$
Here again we have that 
\begin{equation}\label{eq lambda tilde origin}
\tilde\lambda=\lim_{n\to\infty}\lambda_n\tilde\epsilon_n^2=\lim_{n\to\infty}\epsilon_n^{-2}\tilde\epsilon_n^2\in[0,1]
\end{equation}
again thanks to Lemma \ref{Lemma maxima bounds}.
     Then, we may take the limit to deduce that, up to subsequence, there exists $\tilde u$ such that $\tilde u_n\to \tilde u$ in $C_\loc^1([0,+\infty))$ which is a solution to 
    $$\left\{\begin{array}{lll}
    -\tilde u''-\frac{N-1}{r}\tilde u'+\tilde\lambda\ \tilde u = \tilde u^{q-1}, \quad \text{in}\ [0,+\infty)
    \\
    \tilde u\geq 0,\\
    {\tilde u'(0)} = 0.
    \end{array}\right.$$
    In addition, by local uniform convergence we have
    $$\tilde u(\eta)=\lim_{n\to\infty}\tilde u_n\left(\frac{a_n}{\tilde\epsilon_n}\right)=1.$$
Contrary to the previous case, observe that this time we may recover an $N$-dimensional problem by undoing the change of variables: viewing $\tilde u$ as a function in $\R^N$ we obtain a nontrivial radial solution to the problem 
\begin{equation}\label{eq:limit blow up}\left\{\begin{array}{lll}
    -\Delta  \tilde u +  \tilde\lambda\tilde u=  \tilde u^{q-1}, \quad \text{in}\quad \R^N  \\\tilde u \geq 0. 
    \end{array}
        \right.
    \end{equation}
As in the previous case we may prove that the Morse index of the function $\tilde u$ is less then $m^*$. Indeed, suppose by contradiction $\morse(\tilde u) > m^*$. Then there would
  exist linearly independent functions
  $\varphi_1, \dots, \varphi_m \in H^1_\rad(\R^N) \cap C_c^\infty(\R^N)$
  with $m > m^*$ such that    $Q^N_0(\varphi_i; \tilde u)<0$ for every $i \in \{1, \cdots, m\}$. Define the $N$-dimensional functions
\[
\varphi_{i,n}(x):= \tilde \eps_n^{\frac{1}{2}} \varphi_i\left(\frac{x}{\tilde \eps_n}\right).
\]
For every fixed $n$, the functions $\varphi_{1,n},\dots,\varphi_{m,n}$ are linearly independent in $H^1_{\rad}(\R^N)$ and $\supp(\varphi_{i,n})\subseteq B(0,R_i\tilde\epsilon_n)$ where each $R_i$ depends only on the support of $\varphi_i$. In all, using the uniform convergence on compact sets we have
\begin{align*}
    {\frac{1}{\tilde\epsilon_n^{N-1}}}Q_V^N(\varphi_{i,n};u_n)
    &={\frac{1}{\tilde\epsilon_n^{N-1}}}\int_{\R^N} |{\nabla\varphi_{i,n}(x)}|^2 + (V(x)+\lambda_n - \rho_n(q-1)|u_n(x)|^{q-2}
)\varphi^2_{i,n}(x)   dx\\
&={\frac{1}{\tilde\epsilon_n^{N-1}}}\int_{B(0,R_i\tilde\epsilon_n)}|{\nabla\varphi_{i,n}(x)}|^2 + (V(x)+\lambda_n - \rho_n(q-1)|u_n(x)|^{q-2}
)\varphi^2_{i,n}(x)dx\\
&=  \int_{\supp(\varphi_i)}\left[\left|{\nabla\varphi_{i}\left(y\right)}\right|^2 + (\tilde\epsilon_n^2V(\tilde\epsilon_n y+a_n)+\tilde\epsilon_n^2\lambda_n - \rho_n(q-1)|\tilde u_n(y)|^{q-2}
)\varphi^2_{i}\left(y\right)\right]dy\\
&\to\int_{\supp(\varphi_i)}\left[|{\nabla\varphi_{i}(y)}|^2 + \left(\tilde\lambda- (q-1)|\tilde u(y)|^{q-2}
\right)\varphi^2_{i}(y)\right]dy\\
&=Q_0^N(\varphi_i;\tilde u)<0.
\end{align*}
{This proves that $Q_V^N(\varphi_{i,n};u_n)<0$ for $n \in \N$ sufficiently large.} Once again this contradicts the assumption that $\morse(u_n) \leq m^*$ and thus $\morse(\tilde u) \leq m^*$. Recalling that solutions to \eqref{eq:limit blow up} with finite Morse index and $\tilde\lambda=0$ satisfy $u\equiv 0$, see \cite[Theorem 3]{BaLi}, this yields that $\tilde\lambda>0$, proving \eqref{eq:eps n ratio origin}. At this point and since $\tilde u$ is radial, by \cite[Theorem 2]{GiNiNi81} we deduce that, $\tilde{u}$ is decreasing with respect to $|x|.$ This implies that its maximum is reached at the origin, namely that $\eta =0$. Thus \eqref{eq: good ratio origin} holds. This finishes (a').

Consider now the alternative rescaling
\begin{equation}\label{eq:good rescaling2}
    v_n(y)=\lambda_n^{-\frac{1}{q-2}}u_n\left(\frac{y}{\sqrt{\lambda_n}}\right), \quad y\in \R^N.
\end{equation}
    This time we obtain, up to subsequences, that $\{v_n\}$ converges in $C_{\mathrm{loc}}^1(\R^N)$ to a non-negative function $\phi$ satisfying
    \begin{equation}\label{Eq scaled convergence2}\left\{\begin{array}{lll}
    -\Delta  \phi +  \phi=  \phi^{q-1}, \quad \text{in}\quad \R^N    \\
    {\phi(\eta)=\phi(0)=\tilde\lambda^{-\frac{1}{q-2}}}.
    \end{array}
        \right.
    \end{equation}
		As before we can prove that $\phi$ has a finite Morse index. {In particular, $\phi$ is stable outside a compact set and thus $\phi\in H^1(\R^N)$}. Since from (a') there holds $\tilde\lambda\neq 0$, $\phi$ cannot be the trivial solution and in view of Lemma \ref{lemma: morse limit problem} (i) we have proven (b'). Turning to (c'), as a consequence of Lemma \ref{lemma: morse limit problem} (ii), there exists a radial function $\varphi\in C_c^\infty(\R^N)$ such that $Q_0^N(\varphi;\phi)<0$. Then setting
$$\varphi_{n}(x):=  \eps_n^{\frac{1}{2}} \varphi\left(\frac{x}{ \eps_n}\right)$$
we have 
that
\begin{align*}
    {\frac{1}{\tilde\epsilon_n^{N-1}}}Q_V^N(\varphi_n;u_n)
    &={\frac{1}{\tilde\epsilon_n^{N-1}}}\int_{\R^N}\left[|\nabla\varphi_n(x)|^2+(V(x)+\lambda-\rho_n(q-1)|u_n(x)|^{q-2})\varphi_n^2(x)\right]dx
    \\
&=\int_{\R^N}\left[\left|{\nabla\varphi\left(y\right)}\right|^2 + (\epsilon_n^2V(\epsilon_n y+a_n)+\epsilon_n^2\lambda_n - \rho_n(q-1)|v_n(y)|^{q-2}
)\varphi^2\left(y\right)\right]dy\\
&\to\int_{\R^N}\left[|{\nabla\varphi(y)}|^2 + \left(1- (q-1)|\phi|^{q-2}
\right)\varphi^2(y)\right]dy\\
&=Q_0^N(\varphi;\phi)<0.
\end{align*}
Since $\supp (\varphi_n)\subset B(0, a_n + \bar R\tilde\epsilon_n)$ for some fixed $\bar R$ which depends only the support of $\phi$, the statement follows. 
Finally, by local uniform convergence, for every fixed $R>0$ and $s\geq 2$, we have 
$$\lim_{n\to\infty} \lambda_n^{\frac{N}{2}-\frac{s}{q-2}}\int_{B(0, a_n + R\lambda_n^{-1/2})}|u_n(x)|^sdx
=\lim_{n\to\infty}\int_{B(0, a_n \lambda_n^{1/2} +R)}|v_n(y)|^s dy =\int_{B(0,R)}|\phi|^s dy,$$

which proves (d') and ends the proof.
\end{proof}

In this second theorem, considering sequences of solutions to \eqref{L_EQ:bounded approximation} as in Theorem \ref{mainblow}, we show that there exists a finite number of local minima and associated neighbourhoods away from which one has an exponential decay.

\begin{theorem}\label{theo sec blow}
Let $(V0)-(V1)$ hold and $\{(\rho_n,\lambda_n,u_n)\}\in[\frac12,1]\times\R\times H_\rad^1(\R^N)$ be a sequence of solutions to  \eqref{L_EQ:bounded approximation}. Suppose $\rho_n\to 1^-$ and $\lambda_n\to+\infty$. Suppose moreover that the Morse index of the sequence $\{u_n\}$ is uniformly bounded, i.e. there exists $m^*<+\infty$ such that,
$$
\sup_{n\in \N} \morse(u_n) \leq m^*.
$$
\begin{itemize}
    \item [(a)] There exist $M$ sequences $\{a_n^i\} \subset [0, + \infty),$ $i \in \{1, \cdots, M \}$ with $M\le m^*$ such that 
\begin{itemize}
    \item [i)]$\forall i \ne j$ there holds 
    \begin{equation*}
        \lambda_n^{1/2} |a^i_n - a^j_n| \to +\infty ;
    \end{equation*}
  \item[ii)] for every $i\in \{1, \cdots, M \}$ there exists an associated sequence $R_n^i$ such that, $R_n^i\to +\infty$ as $n\to \infty$  and
	\begin{equation}\label{LL1}
 u_n(a_n^i)=\max_{|r-a_n^i|<R_n^i\lambda_n^{-\frac12}}u_n(r);
\end{equation}
    \item [iii)] 
\begin{equation}  \label{eq pseudo decay}  
    \lim_{R\to\infty}\left(\limsup_{n\to\infty}\left[\lambda_n^{-\frac{1}{q-2}}\sup_{d_{n}(r)>R\lambda_n^{-\frac{1}{2}}}u_n(r)\right]\right)=0;
    \end{equation}
		where
		\begin{equation*}
		 d_n(r) := \min_{1\le i \le M} |r-a^i_n|.
		\end{equation*}
\end{itemize}
\item[(b)] There exists $\gamma >0$ such that for all $\epsilon>0$ small enough, there exist $N_\epsilon$ and $R_\epsilon$ such that if $n\geq N_\epsilon$ and $R\geq R_\epsilon$ there holds
\begin{equation}\label{eq far exponential decay}
    u_n(r)\leq\epsilon\lambda_n^{\frac{1}{q-2}}e^{\gamma R}\sum_{i=1}^Me^{-\gamma\lambda_n^\frac12|r-a_n^{i}|}\quad \text{in}\quad \mathcal{A}_{R,n},
\end{equation}
where \begin{equation}\label{eq set Arn}
    \mathcal{A}_{R,n} := \{r \in \R_{\geq 0} \colon d_n(r) \ge R \lambda_n^{-1/2} \}.
\end{equation}

  \end{itemize}
\end{theorem}
\begin{remark}
Thanks to elliptic regularity estimates (c.f. \cite{GiTr}, \cite[Section 5]{EsPe11}), we may also derive a similar bound on the derivatives
\begin{equation}\label{eq far exponential decay derivative}
    u_n'(r)\leq\epsilon\lambda_n^{\frac12+\frac{1}{q-2}}e^{\gamma R}\sum_{i=1}^Me^{-\gamma\lambda_n^\frac12|r-a_n^{i}|}\quad \text{in}\quad \mathcal{A}_{R,n},
\end{equation}
\end{remark}

\begin{remark}
If one closely follows the Step 3 of the proof below, possibly by making $\epsilon$ smaller and $N_\epsilon, R_\epsilon$ larger, one can obtain estimate \eqref{eq far exponential decay} for every $0<\gamma<1$. However, it does not seem possible to recover an
estimate with $\gamma=1$.
\end{remark}

\begin{remark}\label{rem blow up cases}
For a sequence of local maxima $\{a_n^i\} \subset [0, + \infty)$ as considered in Theorem \ref{mainblow} we had to distinguish whether, up to subsequences,  $\lim_{n\to\infty}{a_n^i}\lambda_n^{1/2}$ was finite or not. Theorem \ref{theo sec blow} shows that it is possible to assume that at most one such sequence $\{a_n^i\}$ satisfies $\lim_{n\to\infty}{a_n^i}\lambda_n^{1/2} < \infty$ or, equivalently, satisfies \eqref{eq: blow up origin}. 
\end{remark}
\begin{proof}
To prove part (a), begin by observing that the sequence $\{a_n^1\}$ formed by the global maxima of $u_n$ satisfies \eqref{LL1}. If iii) is also satisfied then taking $M=1$ the claim is proved. If not, we apply the following iterative procedure.

{\it Step 1:
Suppose that for a given $k\in \N$ we have the existence of maxima $a_n^1,\dots,a_n^k$ such that i) and ii) holds but iii) does not. Then there exists an additional  sequence $\{a_n^{k+1}\}$ satisfying i) and ii).}

By assumption there exists $\delta>0$ such that 
$$\limsup_{R\to\infty}\left( \limsup_{n\to+\infty}\left[\lambda_n^{-\frac{1}{q-2}}\max_{d_{n,k}(r)>R\lambda_n^{-\frac{1}{2}}}u_n(r)\right]\right)>4\delta,\quad \mbox{where}\quad d_{n,k}(r) := \min_{1\le i \le k} |r-a_n^i|.$$ 
Then, up to passing to a subsequence, we can fix a $R>0$ as large as we want, so that
    \begin{equation}\label{Eqxn1contrad}
        \lambda_n^{-\frac{1}{q-2}}\sup_{d_{n,k}(r)>R\lambda_n^{-\frac{1}{2}}}u_n(r)\geq 2\delta.
    \end{equation} 
Thanks to \eqref{Eqxn1contrad} and since by the radial nature of the functions {$u_n\in H^1_\rad(\R^N)$} we have $u_n(r)\to 0$ as $r \to\infty$ (see for example \cite[Appendix]{BeLi1}) a sequence of local maximum points 
$\{a_n^{k+1}\}$ must be achieved in the {set} 
$${\mathcal{A}_{k,n}}:=\left\{r\in \R_{\geq 0}\colon \ \min_{1\leq i\leq k}|r-a_n^i|\geq R\lambda_n^{-\frac12}\ \right\}.$$ We claim that the sequence $\{a_n^{k+1}\}$ satisfies i) and ii). To begin with, suppose for contradiction that for some $j\in \{1,\cdots, k \}$ there holds
    $$\lambda_n^{\frac{1}{2}}|a_n^{k+1}-a_n^{j}|
    \to R'\geq R.$$
    We distinguish whether, up to subsequences, $\lambda_n^{\frac12}a_n^j\to +\infty$ or $\lambda_n^{\frac12}a_n^j\to \mu<+\infty.$
    In the former case, let  
    $b_n^{k+1}:=\lambda_n^{\frac12}(a_n^{k+1}-a_n^{j})$ which, up to a subsequence, converges, to a point $\hat b^{k+1}\in \R$ such that $|\hat b^{k+1}|=R'\geq R.$  Now define the scaled unidimensional sequence 
    $$v_n^{j}(s)=\lambda_n^{-\frac{1}{q-2}}u_n\left(\frac{s}{\sqrt{\lambda_n}}+a_n^j\right), \quad s\in{ \left(-R_n^j,R_n^j\right)}, $$ 
which converges, thanks to Theorem \ref{mainblow}, in $C_{\mathrm{loc}}^1(\R)$ to a function ${\phi^{j}\in H^1(\R)}$ which is a solution to \eqref{Eq scaled convergence sphere}. Moreover we have
\begin{equation}\label{eq pointwise sphere}
    2\delta\leq \lambda_n^{-\frac{1}{q    -2}}u_n(a_n^{k+1})=v_n^{j}\left(\lambda_n^\frac12\left(a_n^{k+1}-a_n^{j}\right)\right)=v_n^{j}(b_n^{k+1})\to \phi^{j}(\hat b^{k+1}).
\end{equation}
    In the latter case, the sequence 
    $b_n^{k+1}:=\lambda_n^{\frac12}(a_n^{k+1}-a_n^{j})>0$ which, up to having chosen $R > \mu$, directly converges to $R'$. Analogously, 
    define the radial scaled sequence
    $$v_n^j(y)=\lambda_n^{-\frac{1}{q-2}}u_n\left(\frac{y}{\sqrt{\lambda_n}}\right), \quad y\in B\left(0,{R_n^j}\right)$$
    which converges in $C_{\mathrm{loc}}^1(\R^N)$ to a function ${\phi^{j}\in H^1(\R^N)}$ which is a solution to \eqref{Eq scaled convergence origin}. This time there holds 
    \begin{equation}\label{eq pointwise origin}
    2\delta\leq \lambda_n^{-\frac{1}{q    -2}}u_n(a_n^{k+1})=v_n^{j}\left(\lambda_n^\frac12{a_n^{k+1}}\right)=v_n^{j}(b_n^{k+1}+{\lambda_n^{\frac12}}a_n^j)\to \phi^{j}(R'+{\mu}).
    \end{equation}
By choosing $R>0$ sufficiently large so that {$|\phi^j(x)|<\delta$ for $|x|\geq R$}  we obtain a contradiction either to \eqref{eq pointwise sphere} or to \eqref{eq pointwise origin}. Such contradiction arises from the fact that, referring to Lemma \ref{lemma: morse limit problem} (i) or to Lemma \ref{2_lemma lagrange phase plane}, 
in both cases $\phi^j$ decays to $0$ at infinity. In conclusion, the newly constructed sequence $\{a_n^{k+1}\}$ satisfies i). 

In order to prove ii) we define
$$R_n^{k+1}:=\frac12\lambda_n^{\frac{1}{2}}{d_{n,k}(a_n^{k+1})}.$$ Note that $R_n^{k+1}\to \infty$ as $n\to\infty$ thanks to i).
Since from the construction performed in i), $\{a_n^{k+1}\}$ is a sequence of local maxima lying in $\mathcal{A}_{k,n},$ in order to conclude it suffices to prove that, for $n$ sufficiently large,
\begin{equation}\label{eq new sequence local max}
    \left\{r\in \R_{\geq 0}\colon\ |r-a_n^{k+1}|<R_n^{k+1}\lambda_n^{-\frac12}\right\}\subseteq {\mathcal A_{k,n}}.
\end{equation}
This is equivalent to showing that $|r-a_n^{k+1}|<R_n^{k+1}\lambda_n^{-\frac12}$ implies $|r-a_n^{i}|\geq R\lambda_n^{-\frac12}$  for every $i\in [1,\dots, k]$. Using that $d_{n,k}(a_n^{k+1})\leq |a_n^{k+1}-a_n^j|$ for every $j\in [1,\dots, k]$ and the definition of $R_n^{k+1}$ we obtain, taking $n$ sufficiently large
\begin{align*}
    |a_n^i-r|\geq |a_n^{i}-a_n^{k+1}|-|a_n^{k+1}-r|\geq 2R_n^{k+1}\lambda_n^{-\frac12}-R_n^{k+1}\lambda_n^{-\frac12}=R_n^{k+1}\lambda_n^{-\frac12}\geq R\lambda_n^{-\frac12}.
\end{align*}
As a result, \eqref{eq new sequence local max} is proven and so is ii).

To sum up, Step 1 gives an iterative process which constructs a new sequence of local maxima whenever iii) is not satisfied. If with this new sequence iii) is satisfied we are done. If not we reapply Step 1. The next step is dedicated to showing that this iterative process must stop if we already have the existence of $m^*$ sequences.

{\it Step 2: If there already exist $a_n^1,\dots,a_n^{m^*}$ sequences satisfying i) and ii), then iii) holds.} 

    Suppose for contradiction iii) does not hold. Following the arguments of Step 1 we deduce the existence of an $(m^*+1)$-th sequence  $a_n^{m^*+1}$ 
    satisfying i) and ii). As such, from the iterative procedure we have $(m^*+1)$ sequences $\{a_n^{i}\}_{i=1}^{m^*+1}$ satisfying i) and ii) and thus each of them fall under the hypothesis of Theorem \ref{mainblow}. As a consequence, using conclusions (c) or (c') depending on the case, more specifically either \eqref{eq: morse sequence sphere} or \eqref{eq:morse sequence origin}, we obtain $(m^*+1)$ sequences $\{\phi_n^i\}_{i=1}^{m^*+1}$ in $C_c^\infty(\R^N)$ satisfying $$\supp(\phi_n^i)\subseteq A(a_n^i,\bar R^i\lambda_n^{-\frac12})$$
    for some $\bar R^i>0$ and such that 
    $$Q_V^N(\phi_n^i,u_n)<0, \quad \text{for all}\ i\in \{1,\cdots,m^*+1\}.$$
Using now that the sequences $\{a_n^i\}_{i=1}^{m^*+1}$ satisfy i), we can make the supports of the $\phi_n^i$ disjoint between themselves. In turn this would imply that $m(u_n)\geq m^*+1$ which is a contradiction. The proof of (a) is now complete.

\textit{Step 3: Proof of (b)}. From a) iii) we have that, given $\epsilon>0$ there exists $R_\epsilon$ such that if $R\geq R_\epsilon$ then
    \begin{equation}\label{Eqbound in Anr}
        \max_{r\in \mathcal{A}_{R,n}}u_n(r)\leq \frac14\epsilon\lambda_n^{\frac{1}{q-2}}
    \end{equation}
for $n$ sufficiently large. Without loss of generality we may assume that
\begin{equation}\label{epsilon bound}
    \epsilon^{q-2}\leq 
4^{q-3}. 
\end{equation} Possibly by making $R_\epsilon$ larger we may assume
\begin{equation}\label{eq R bound}
    \frac{N-1}{R}\leq \frac12.
\end{equation}
Moreover, we fix $N_\epsilon$ sufficiently large so that for all $n\geq N_\epsilon$ there holds 
{
\begin{equation}\label{potential bound Anr}
4{|\underline{V}|}\leq \lambda_n.
\end{equation}
}
Our aim  here is to make use of the comparison principle to establish convenient bounds on $u_n$. To this end, we define the linear operator $L_n(u)=-\Delta u+\lambda_nV_n(r)u$ with
$$V_n(r)=\frac{V(r)}{\lambda_n}+1-\frac{(u_n(r))^{q-2}}{\lambda_n}
.$$
Since $u_n$ is a solution to \eqref{L_EQ:bounded approximation} we have $L_n(u_n)=0.$ Under \eqref{Eqbound in Anr}-\eqref{potential bound Anr}
 we may check for all $n\geq N_\epsilon$ and $R\geq R_\epsilon$ (note $R_\epsilon$ is also dependent on $N_\epsilon$) that
\begin{equation}\label{Eq V_n bound}\begin{split}
    V_n(r)&\geq 1-\frac14-\frac{\epsilon^{q-2}}{4^{q-2}}
    \geq \frac12,\quad  \quad r\in \mathcal{A}_{R,n}.
    \end{split}
\end{equation}
Consider now, for $i\in \{1,\cdots, k\}$ and {$\gamma \in (0,  \frac12)$} the sequences of $N$-dimensional functions
$$\varphi_{n}^{i}(|x|)=\varphi_{n}^{i}(r)=e^{-\gamma\lambda_n^\frac12|r-a_n^{i}|}.$$
One can see through direct calculation that the functions $\{\varphi_{n}^{i}\}$ are radial solutions to 
{$$-\Delta \varphi=-\varphi''(r)-\frac{N-1}{r}\varphi'(r)=-\gamma^2\lambda_n\varphi_{n}^{i}-\frac{(N-1)}{r}\frac{r-a_n^i}{|r-a_n^i|}\gamma \lambda_n^{\frac12}\varphi_n^i,\quad r\in \mathcal{A}_{r,n}.$$
If $r\in \mathcal{A}_{r,n}$ is such that $0\leq\ r\leq a_n^i-R\lambda_n^{-\frac12}$ then $-\frac{(N-1)}{r}\frac{r-a_n^i}{|r-a_n^i|}\geq 0$. If on the contrary $r\geq a_n^i+R\lambda_n^{-\frac12}$, then
$$-\left(\frac{1}{\lambda_n}\right)\frac{(N-1)}{r}\frac{r-a_n^i}{|r-a_n^i|}\gamma\lambda_n^{\frac12}\geq -\frac{(N-1)}{a_n^i\lambda_n^{\frac12}+R}\gamma\geq -\frac{(N-1)}{R}\gamma,$$
so that, by means of \eqref{eq R bound},
\begin{equation}\label{eq derivative term estimate}
{
-\frac{(N-1)}{r}\frac{r-a_n^i}{|r-a_n^i|}\gamma\lambda_n^{\frac12}}\geq -\frac{(N-1)}{R}\gamma\lambda_n\geq -\frac{1}{2}\gamma\lambda_n
\end{equation}
for all $r\in \mathcal{A}_{R,n}.$
As a result, using \eqref{Eq V_n bound}, and \eqref{eq derivative term estimate}
\begin{equation}\label{Eq L_n bound} L_n(\varphi_{n}^{i})=-\gamma^2\lambda_n\varphi_{n}^{i}-\frac{(N-1)}{r}\frac{r-a_n^i}{|r-a_n^i|}\gamma \lambda_n^{\frac12}\varphi_n^i+\lambda_n V_n(r)\varphi_{n}^{i}\geq\lambda_n\varphi_{n}^{i}(V_n(r)-\gamma^2-\frac{\gamma}2)\geq 0,\quad  \text{in}\ \mathcal A_{R,n}.
\end{equation}}
 Now observe that by linearity of the operator and \eqref{Eq L_n bound} there holds
\begin{equation}\label{comp ppl operator estimate}
    L_n(\epsilon\lambda_n^{\frac{1}{q-2}}e^{\gamma R}\varphi_{n}^{i}- u_n)=\epsilon\lambda_n^{\frac{1}{q-2}}e^{\gamma R}L_n(\varphi_{n}^{i})\geq 0 .
\end{equation}
  Furthermore, by \eqref{Eqbound in Anr}, we have for all $r=\{ a_n^{i}\pm{R\lambda_n^{-{1}/{2}}}\}\subset \partial\mathcal{A}_{R,n}$ that
\begin{align}
    \epsilon\lambda_n^{\frac{1}{q-2}}e^{\gamma R}\varphi_{n}^{i}- u_n
    &\geq 
    \epsilon\lambda_n^{\frac{1}{q-2}}e^{\gamma R}e^{-\gamma \lambda_n^\frac12(R\lambda_n^{-\frac12})}- u_n
    \geq \epsilon\lambda_n^{\frac{1}{q-2}}- u_n\geq \frac14 \epsilon\lambda_n^{\frac{1}{q-2}}- u_n\geq 0.\label{comp ppl boundary estimate}
\end{align}
In view of \eqref{comp ppl operator estimate} and \eqref{comp ppl boundary estimate} we can use the comparison principle between the functions $u_n$ and $\epsilon\lambda_n^\frac{1}{q-2}e^{\gamma R}\varphi_{n}^{i}$ to establish
$$u_n\leq \epsilon\lambda_n^\frac{1}{q-2}e^{\gamma R}\varphi_{n}^{i}\quad \text{on}\quad [0,+\infty) \backslash (a_n^{i}-{R\lambda_n^{-{1}/{2}}}, a_n^{i}+{R\lambda_n^{-{1}/{2}}}).$$
Repeating the same argument using the function $\varphi_n=\sum_{i=1}^M\varphi_n^{i}(x)$ instead we obtain the desired estimate
$$u_n(r)\leq \epsilon\lambda_n^{\frac{1}{q-2}}e^{\gamma R}\varphi_n =\epsilon\lambda_n^{\frac{1}{q-2}}e^{\gamma R}\sum_{i=1}^Me^{-\gamma\lambda_n^\frac12|r-a_n^{i}|}\quad \text{on}\quad \mathcal{A}_{R,n}.$$
\end{proof}

From the previous Theorem \ref{theo sec blow} we have obtained the existence of a finite number of sequences of local maxima away from which an exponential decay holds. In combination with Theorem \ref{mainblow}, these lead to the formation of maxima which concentrate at spheres or of spike layers, depending on whether \eqref{eq: blow up on sphere} or \eqref{eq: blow up origin} holds respectively. 

As an additional property we now prove that the maxima that we have obtained are in fact the only possible points of local maximum. 
Our proof is inspired by \cite[Corollary 3.2]{EMSS}.
\begin{corollary}\label{all maximum points}
    Let $(V0)-(V1)$ hold and $\{(\rho_n,\lambda_n,u_n)\}\in[\frac12,1]\times\R\times H_\rad^1(\R^N)$ be a sequence of solutions to  \eqref{L_EQ:bounded approximation}. Suppose $\rho_n\to 1^-$ and $\lambda_n\to+\infty$. Suppose moreover that the Morse index of the sequence $\{u_n\}$ is uniformly bounded, i.e. there exists $m^*<+\infty$ such that,
$$
\sup_{n\in \N} \morse(u_n) \leq m^*.
$$
Up to subsequences and for $n \in \N$ sufficiently large, the points $\{a_n^i\} \subset [0, + \infty),$ $i \in \{1, \cdots, M \}$, with $M\le m^*$ obtained in Theorem \ref{theo sec blow} are the \emph{only} points of local maximum of $u_n$.

\end{corollary}
\begin{proof}
{
For fixed $R$ large enough and $n \geq n(R)$ we have, in view of \eqref{Eqbound in Anr}, \eqref{epsilon bound} and \eqref{potential bound Anr},
    \begin{align*}
        -\Delta u_n=u_n[\rho_nu_n^{q-2}-\lambda_n-V(x)]
        \leq u_n\left[\lambda_n\left(\left(\frac\epsilon 4\right)^{q-2}-1\right)-\underline{V}
        \right]
        <0.
    \end{align*}
Hence, any sequence of (radial) local maxima $\{b_n\}$ of $u_n$, for $n \in \N$ large, satisfies
$$d_n(b_n) = \min_{1\le i \le M} |b_n-a^i_n|\leq R\lambda_n^{-\frac12}.$$
	
Suppose first that \eqref{eq: blow up on sphere} holds. Take $c_n=\lambda_n^{\frac12}|b_n-a_n^i|$ and recall the rescaled functions $v_n$ from \eqref{eq:good rescaling sphere}.
By construction, $c_n$ is a local maximum of $v_n$. Now, since $v_n\to \phi$ in $C^1_{\loc}(\R)$ where $\phi$ is a solution to \eqref{Eq scaled convergence sphere} and since $0$ is the only local maximum of $\phi$, necessarily $c_n\to 0$. As a result, 
\begin{align*}
    -v_n''(c_n)&=-V(c_n)\lambda_n^{-1}v_n(c_n)-v_n(c_n)+\rho_nv_n(c_n)^{q-1}\\
    &= -\phi(0)+\phi(0)^{q-1}+o_n(1)=\left(\frac2{q}\right)^{-\frac{1}{q-2}}\left(\frac{q}{2}-1\right)+o_n(1)>0,
\end{align*}
and we conclude that the $c_n$ are in fact strict local maxima of $v_n$ for $n$ sufficiently large. This implies that 
$\{b_n\}$ is a sequence of \emph{strict} local maxima \emph{in the radial direction}. Focusing on the associated unidimensional profile of $u_n$, $b_n$ is a strict local maximum and $a_n^i$ is a local maximum, so in between there must be a point of local minimum $s_n$. However, defining $t_n:=\lambda_n|s_n-a_n^i|$ it is clear that $t_n\to 0$ and that $-v_n''(t_n)>0$ by the same calculation as before. This contradicts the fact that $t_n$ is a local minimum of $v_n$, discarding \eqref{eq: blow up on sphere} as a possibility.
}

{
Next suppose \eqref{Eq scaled convergence origin} holds.
As before, take
$c_n=\lambda_n^{\frac12}b_n$ and recall the  rescaled functions $v_n$ from \eqref{eq:good rescaling origin}. By construction, $c_n$ is a (radial) local maximum of $v_n$. Moreover, since $v_n\to \phi$ in $C^1_\loc(\R^N)$ and the origin is the only point of maximum of $\phi$, once again by the standard result \cite[Theorem 2]{GiNiNi81}, it is forced that $c_n\to 0$. 
Now, using that $\phi\in H^1_{\rad}(\R^N)$ by \eqref{Eq scaled convergence origin}, it is a classical result, see for example \cite[Corollary B.11, Chapter B]{Ta}, that
$\phi(0)\in(1,+\infty)$.
} Working with the associated unidimensional equation of $v_n$ we deduce
\begin{align*}
    -v_n''(c_n)&=-V(c_n)\lambda_n^{-1}v_n(c_n)-v_n(c_n)+\rho_nv_n(c_n)^{q-1}\\
    &\geq -\phi(0)+\phi(0)^{q-1}+o_n(1)>0,
\end{align*}
so that $c_n$ is a strict local maximum of $v_n$ (in the radial direction). We finish this case in an analogous way by deducing the existence of a local minimum $s_n$ of $u_n$ which ultimately leads to a contradiction.
\end{proof}

\begin{remark}
In \cite[Remark 3.8]{CaGaJeTr25}, an example is constructed in which the number of selected blow-up points in a theorem analogous to Theorem \ref{theo sec blow} is strictly  less than the number of local maxima of the absolute value of solutions. However, \cite{CaGaJeTr25} considers a problem set on a graph and solutions that possibly change signs.
\end{remark}

We will now proceed to prove Theorem \ref{theo maxima to 0} and discard the existence of any blow-up sequences. To do so, we first establish an additional relation.

\begin{lemma}\label{one Pohozaev}[Poho\v{z}aev type identity] Let $u\in H^1_\loc([0, +\infty))$, $V\in L^1_{\loc}([0,+\infty))$, and let $u$ be a positive solution to 
\begin{equation}\label{radial sphere}
\begin{cases}
-u''(r)-\frac{N-1}{r}u'(r)+V(r)u(r)+\lambda u(r)=\rho u(r)^{q-1}, &\text{in}\ [0,+\infty);\\
u'(0) =0
\end{cases}
\end{equation}
with $q>2$. Then, for every interval $[a,b] \subset [0, +\infty)$ there holds 
\begin{equation}\label{eq pozohaev pseudo}
    \begin{split}
        \rho\left(N-\frac32-\frac1{q}\right)\int_a^bu^q=\ &\left(N-2\right)\int_a^b\lambda u^2+\left(\frac32-N\right)\left(\frac{N-1}{2}\right)\int_a^b \frac{u^2}{r^2}\\&-\left(\frac32-N\right)\int_a^bVu^2+\int_a^b V(r)ruu'\\
				&
        +\left[\left(\frac32-N\right)uu'+\left(\frac32-N\right)\left(\frac{N-1}{2}\right)\frac{u^2}{r}+\frac{\lambda}{2}u^2r - \frac{\rho}{q}u^qr - \frac12 r(u')^2\right]_a^b.
    \end{split}
\end{equation}

\end{lemma}
\begin{remark}
    Recall that the radial profile of solutions to \eqref{L_EQ:bounded approximation}, 
satisfy \eqref{radial sphere}. 
\end{remark}

\begin{proof}
    To begin with, multiplying \eqref{radial sphere} by $ru'$ and integrating we obtain
    \begin{equation*}
        -\int_a^b u''u'r-\int_a^b\frac{N-1}{r}(u')^2r=-\int_a^bV(r)ruu'-\lambda\int_a^buu'r+\rho\int_a^bu^{q-1}u'r.
    \end{equation*}
    Performing some integration by parts then yields
     \begin{equation*}
     \begin{split}
        \frac12\left(\int_a^b (u')^2-\left[{(u')^2}r\right]^b_a\right)&-({N-1})\int_a^b(u')^2 \\&=-\int_a^bV(r)ruu'+\frac{\lambda}2\left(\int_a^bu^2-[u^2r]_a^b\right)
        -\frac{\rho}{q}\left(\int_a^bu^{q}-[u^qr]_a^b\right)
        \end{split}
    \end{equation*}
    which, after some rearrangement gives
 \begin{equation}\label{eq pozo 1}
     \begin{split}
        \left(\frac32-N\right)\int_a^b (u')^2-\left[\frac{r}{2}{(u')^2}\right]^b_a=-\int_a^bV(r)ruu'+\frac{\lambda}2\left(\int_a^bu^2-[u^2r]_a^b\right)
        -\frac{\rho}{q}\left(\int_a^bu^{q}-[u^qr]_a^b\right).
        \end{split}
    \end{equation}
On the other hand, multiplying \eqref{radial sphere} by $u$ we obtain
\begin{equation*}
        -\int_a^b u''u=\int_a^b\frac{N-1}{r}u'u-\int_a^bV(r)u^2-\lambda\int_a^bu^2+\rho\int_a^bu^{q}.
    \end{equation*}
Integration by parts this time yields
\begin{equation}\label{eq pozo 2}
     \begin{split}
        \int_a^b (u')^2-\left[{u'u}\right]^b_a =\frac{N-1}{2}\left(\int_a^b\frac{u^2}{r^2}+\left[\frac{u^2}r\right]_a^b\right)-\int_a^bV(r)u^2-\lambda\int_a^bu^2+\rho\int_a^bu^{q}.
        \end{split}
    \end{equation}
Substituting \eqref{eq pozo 2} into \eqref{eq pozo 1} yields \eqref{eq pozohaev pseudo}.
\end{proof}
The next lemma provides a useful estimate on one of the terms in \eqref{eq pozohaev pseudo}.

\begin{lemma} Let $\mathbf{u}\in H^1_\rad(\R^N)$, and let $u$ be its unidimensional profile, i.e. $\mathbf{u}(x)=u(|x|)$ for all $x\in \R^N$. Let $V\in L^\infty(\R^N)$ and $0<a$. Then, for all $a\leq b\leq +\infty$,
\begin{equation}\label{potential estimate}
    \left|\int_{a}^{b}V(r)ruu'dr\right|\leq  \frac{\|V\|_{\infty}}{|S^{N-1}|}  \frac{1}{a^{N-2}}\|\mathbf u\|_{L^2(\R^N)}\|\nabla \mathbf u\|_{L^2(\R^N)}
\end{equation}

where $|S^{N-1}|$ denotes the area of the unit sphere in $\R^N$.
\end{lemma}
\begin{proof}
    Due to the radial nature of $\mathbf{u}$, it satisfies
$$\nabla \mathbf{u}(x)=u'(|x|)\cdot\frac{x}{\|x\|}$$
and thus
$$|\nabla \mathbf{u}(x)|=|u'(|x|)|.$$ With this relation in mind, letting 
$$A((a,b)):=\{x\in\R^N\colon |x|\in(a,b)\}$$
there holds 
\begin{align*}
    \left|\int_a^bV(r)ruu'dr\right| & \leq  \int_a^b\left|V(r)ruu'\right|dr \\
    & \leq \|V\|_{\infty}\int_a^b r|uu'|dr\\
		& \leq \|V\|_{\infty}\frac{1}{a^{N-2}}\int_a^b r^{N-1}|uu'|dr \\
	\	& \leq\ \|V\|_{\infty}\frac{1}{a^{N-2}}\frac{1}{|S^{N-1}|} \int_{A((a,b))} |u(|x|)||u'(|x|)|dx \\
    & \leq  \|V\|_{\infty}\frac{1}{a^{N-2}}  \frac{1}{|S^{N-1}|} \|\mathbf{u}\|_{L^2(\R^N)}\|\nabla \mathbf{u}\|_{L^2(\R^N)}.
\end{align*}

\end{proof}

We finally  give the 
\begin{proof}[Proof of Theorem \ref{theo maxima to 0}.]

We make a proof by contradiction, assuming that $\lambda_n \to \infty$ at least for a subsequence.
From Theorem \ref{theo sec blow} there exist at most $M$ sequences $a_n^1,\dots,a_n^M$ with $M\leq m^*$ satisfying i) and ii). Passing to subsequences, without loss of generality we may rename the sequences so that they are ordered, i.e. $a_n^i\leq a_n^j $ if $i\leq j.$ 
Our procedure to proving the result will be divided into two main steps. First, we will suppose that a given sequence $\{a_n^i$\} satisfies $a_n^i\lambda_n^{1/2}\to +\infty$ and obtain a contradiction. To do so, we implement the sequence $\{u_n\}$ into \eqref{eq pozohaev pseudo} with a carefully chosen sequence of intervals $I_n$ containing the selected sequence $\{a_n^i\}$.
Once we have discarded the possibility that $a_n^i\lambda_n^{1/2}\to +\infty$, the second step will then discard the possibility that $\{a_n^i\lambda_n^{1/2}\}$ remains bounded. Note already that from Remark \ref{rem blow up cases} we know that if such sequence  exists then it is unique.  

Before delving into the procedure an observation is in order. From \eqref{ALL1} and \eqref{eq:gradient boundedness Rn}, we obtain that, for $n\in \N$ sufficiently large,
$$\left(\frac{q-2}{2q}\right)\|\nabla u_n\|_{L^2(\R^N)}^2\leq E_{\rho_n}(u_n) +\frac{\lambda_n\mu_2^2}{q}- \left(\frac{q-2}{2q}\right) \underline{V} \mu_2^2\leq \frac{3\mu_2
    ^2}{q}\lambda_n
    $$
    or, equivalently
\begin{equation}\label{eq gradient estimate}
    \|\nabla u_n\|_{L^2(\R^N)}\leq \mu_2\sqrt{\frac{6\lambda_n}{q-2}}.
\end{equation}

\textit{Step 1: discarding the possibility \eqref{eq: blow up on sphere}}.

Here we work under the assumption that
\begin{equation}\label{goinfinity}
\lambda_n^{\frac{1}{2}}a_n^i \to + \infty,
\end{equation}
for some given $i\in [1,\dots,M]$.
For $i \leq j \leq M,$ let
$$\overline{k}_j=\lim_{n\to \infty}\frac{a_n^{j}}{a_n^i}$$
(note that, passing to subsequences, such limits exist) and let
$$\alpha=\max \Big\{i \leq j \leq M \colon \overline{k}_j\neq+\infty \Big\}.$$ 
We set $\overline{K} = \overline{k}_{\alpha}$. Recall that the sequences are ordered and so $\overline{K} \in [1, + \infty)$.

Also let, for $1 \leq j \leq i$,
$$\underline{k}_j=\lim_{n\to\infty}\frac{a_n^{j}}{a_n^i}$$
and
$$\beta=\min \Big\{1 \leq j \leq i, \colon \underline{k}_j >0\Big\}.$$ 
We set $\underline{K} = \underline{k}_{\beta}$. Note that $\underline{K} \in (0,1].$

Finally, define
\begin{equation}\label{keyset}
I_n=\Big[\frac{1}{2} \underline{K}a_n^i, \frac{3}{2} \overline{K} a_n^i \Big]\supset[a_n^i-R\lambda_n^{-\frac12},a_n^i+R\lambda_n^{-\frac12}].
\end{equation}
Also let
\begin{equation}\label{setJi}
    J_i=\{j\in \{\beta,\dots, \alpha\} \}.
\end{equation}
    
As anticipated, our aim is to substitute $[a,b]:=I_n, u=u_n$ into \eqref{eq pozohaev pseudo} after multiplication by $\lambda_n^{\frac12-\frac{q}{q-2}}$. To do so we will analyze each of the terms separately. Let us focus first on the boundary terms. We shall prove that there can all be neglected.  In this aim some estimates are in order. 

There holds, for $n \in \N$ sufficiently large, that $a_n^{\alpha}\leq \frac{5}{4} \overline{K}a_n^{i}$ and so 
\begin{equation}\label{L1}
\frac{3}{2} \overline{K}a_n^{i}-a_n^j \geq \frac{3}{2} \overline{K}a_n^{i} -a_n^\alpha\geq \frac14 \overline{K}a_n^i, \quad \mbox{for } i \leq j\leq\alpha.
\end{equation}
Also, if $\alpha <M$, we have, by definition of $\alpha$, that
$$\lim_{n\to\infty}\frac{a_n^{\alpha+1}}{a_n^i}=+\infty$$ and so for $n \in \N$ sufficiently large $a_n^{\alpha+1}\geq \frac74 \overline{K}a_n^i$.
Then, 
\begin{equation}\label{L2}
a_n^j -\frac{3}{2} \overline{K}a_n^{i} \geq  a_n^{\alpha +1} - \frac{3}{2} \overline{K}a_n^{i} \geq \frac14 \overline{K}a_n^i, \quad \mbox{for } \alpha +1 \leq j\leq M.
\end{equation}
Note that if $\alpha =M$, then $\eqref{L1}$ directly holds for $i \leq j \leq M.$  Globally thus, in view of \eqref{L1} and \eqref{L2},
\begin{equation}\label{L3}
\Big|\frac{3}{2} \overline{K}a_n^{i}-a_n^j  \Big| \geq \frac14 \overline{K}a_n^i, \quad \mbox{for } i \leq j\leq M
\end{equation}
and \eqref{L3} obviously extends to all $j \in \{1, \cdots, M \}.$  The estimate \eqref{L3} along with \eqref{eq far exponential decay} and \eqref{eq far exponential decay derivative} enable to deduce that
\begin{align*}
\left|\left.\lambda_n^{\frac12-\frac{q}{q-2}}u_nu'_n\right|_{r=\frac{3}{2} \overline{K}a_n^{i}}\right|&\leq \lambda_n^{\frac12-\frac{q}{q-2}}\left(\lambda_n^{\frac{1}{q-2}+{\frac{1}{q-2}+\frac12}}\right)\epsilon^2 e^{2\gamma R}\left(\sum_{j=1}^M e^{-\gamma\lambda_n^\frac12|\frac{3}{2} \overline{K}a_n^{i}-a_n^j|}\right)^2\\&
   \leq \lambda_n^{0}\cdot\epsilon^2 e^{2\gamma R}\left(\sum_{j=1}^M e^{- \gamma\frac{\overline{K}}{4}\lambda_n^\frac12 a_n^i} \right)^2
   \to 0,
\end{align*}
where the convergence to $0$ follows from  \eqref{goinfinity}. Similarly
\begin{equation}\label{pohozaev boundary estimates stationary}
\begin{split}
   \left.\lambda_n^{\frac12-\frac{q}{q-2}}\frac{u_n^2}{r}\right|_{r=\frac{3}{2} \overline{K}a_n^{i}}
  &=\frac{\lambda_n^{-\frac12}}{\frac{3}{2} \overline{K} a_n^i}\epsilon^2 e^{2\gamma R}\left(\sum_{j=1}^M e^{- \gamma\frac{\overline{K}}{4}\lambda_n^\frac12 a_n^i} \right)^2\to 0;\\
   \left.\lambda_n^{\frac12-\frac{q}{q-2}}\lambda_n u_n^2r\right|_{r=\frac{3}{2} \overline{K}a_n^{i}}&\leq \lambda_n^{\frac12}(\frac{3}{2} \overline{K} a_n^i)\epsilon^2 e^{2\gamma R}\left(\sum_{j=1}^M e^{- \gamma\frac{\overline{K}}{4}\lambda_n^\frac12 a_n^i} \right)^2\to 0;
   \\
    \left.\lambda_n^{\frac12-\frac{q}{q-2}}u_n^qr\right|_{r=\frac{3}{2} \overline{K}a_n^{i}}&\leq\lambda_n^{\frac12}(\frac{3}{2} \overline{K} a_n^i)\epsilon^2 e^{q\gamma R}\left(\sum_{j=1}^M e^{- \gamma\frac{\overline{K}}{4}\lambda_n^\frac12 a_n^i} \right)^q \to 0;\\
    \left.\lambda_n^{\frac12-\frac{q}{q-2}}(u_n')^2r\right|_{r=\frac{3}{2} \overline{K}a_n^{i}}&\leq\lambda_n^{\frac12}(\frac{3}{2} \overline{K}a_n^{i})
    \epsilon^2 e^{2\gamma R}\left(\sum_{j=1}^M e^{- \gamma\frac{\overline{K}}{4}\lambda_n^\frac12 a_n^i} \right)^2
    \to 0.
    \end{split}
\end{equation}
To tackle the boundary terms at $r = \frac{1}{2}\underline{K} a_n ^i$ observe that, for $n \in \N$ sufficiently large,
$\frac{3}{4}\underline{K}a_n^i\leq a_n^{\beta}$
and so 
\begin{equation}\label{L4}
a_n^j - \frac{1}{2} \underline{K}a_n^{i} \geq a_n^\beta - \frac{1}{2} \underline{K}a_n^{i} \geq \frac14 \underline{K}a_n^i, \quad \mbox{for } \beta \leq j\leq i.
\end{equation}
Also, if $\beta >1$, we have, by definition of $\beta$, that
$$\lim_{n\to\infty}\frac{a_n^{\beta -1}}{a_n^i}= 0$$ and so, for $n \in \N$ sufficiently large, $a_n^{\beta -1}\leq \frac14 \underline{K}a_n^i$.
Then, 
\begin{equation}\label{L5}
\frac{1}{2} \underline{K}a_n^{i} - a_n^j  \geq   \frac{3}{2} \underline{K}a_n^{i} - a_n^{\beta - 1} \geq \frac14 \underline{K}a_n^i, \quad \mbox{for } 1 \leq j\leq \beta -1.
\end{equation}
Note that if $\beta =1$, then $\eqref{L4}$ directly holds for $1 \leq j \leq i.$  Globally thus we get from \eqref{L4} and \eqref{L5} that
\begin{equation}\label{L6}
\Big|\frac{1}{2} \underline{K}a_n^{i}-a_n^j  \Big| \geq \frac14 \overline{K}a_n^i, \quad \mbox{for } 1 \leq j \leq i.
\end{equation}
Also \eqref{L6} obviously extends to all $j \in \{1, \cdots, M \}.$ In view of \eqref{L6} and still using \eqref{goinfinity} we can, just like for the upper bounds, show that each of the boundary terms at $r = \frac{1}{2} \underline{K}a_n^{i}$ goes to $0$ as $n \to \infty$.

We now turn to the integral terms in \eqref{eq pozohaev pseudo}. First let us show that the one involving the product $u_n u_n'$ also converges to $0$. Using \eqref{potential estimate},  \eqref{eq gradient estimate} and taking \eqref{eq mass energy unif bound} into account we get, for some constant $C>0$,
\begin{align*}
    \left|\lambda_n^{\frac{1}{2}-\frac{q}{q-2}}\int_{\frac{1}{2}\underline{K} a_n^i}^{\frac{3}{2}\overline{K} a_n^i}V(r)ru_nu_n'dr\right|&\leq \lambda_n^{\frac{1}{2}-\frac{q}{q-2}} \frac{\|V\|_{\infty}}{|S^{N-1}|} \frac{1}{(\frac{1}{2}\underline{K} a_n^i)^{N-2}}\|u_n\|_{L^2(\R^N)}\|\nabla u\|_{L^2(\R^N)}
    \\
    &\leq \lambda_n^{\frac{1}{2}-\frac{q}{q-2}} \frac{\|V\|_{\infty}}{|S^{N-1}|} \frac{1}{(\frac{1}{2}\underline{K} a_n^i)^{N-2}}\mu_2\left(\mu_2\lambda_n^{\frac12}\sqrt{\frac{6}{q-2}}\right)
    \\
    & \leq C \, \lambda_n^{-\frac{2}{q-2}}  (a_n^i)^{2-N}\to 0.
\end{align*}
Indeed if $N=2$ (or if $a_n^i \not \to 0$), we immediately deduce that the right hand side tends to $0$. If $N \geq 3$, we use that \eqref{goinfinity} gives, for $n \in \N$ sufficiently large, that $a_n^i \geq \lambda_n^{- \frac{1}{2}}$. Then taking into account that $q < 2^* = \frac{2N}{N-2}$ the conclusion also holds.

Regarding the rest of the integral terms in \eqref{eq pozohaev pseudo}, we will split them between what happens close and far from the points of maxima, that is, we make the decomposition
\begin{equation*}
I_{n}=[\frac{1}{2} \underline{K} a_n^i, \frac{3}{2}\overline{K} a_n^i ]=\left(I_{n}\cap\bigcup_{j\in {J}_i}(a_n^j-R\lambda_n^{-\frac12},a_n^j+R\lambda_n^{-\frac12})\right)\cup (I_{n} \cap\mathcal A_{R,n})
\end{equation*}
and proceed to study each of the terms individually. Define, for $j\in {J}_i$ (refer to \eqref{setJi}) and similarly to \eqref{eq:good rescaling sphere}, the functions (for simplicity in the calculations below we omit the dependence in $j$ since it does not play a role)
\begin{equation}\label{rescaling proof}
    v_n(s):=v_n^{j}(s)=\lambda_n^{-\frac{1}{q-2}}u_n\left(\frac{s}{\sqrt{\lambda_n}}+a_n^j\right).
\end{equation}
 Then, close to the points of maxima for all $j\in {J}_i$ we have, as $n\to +\infty$,
\begin{equation*}
    \begin{split}
    \lambda_n^{\frac12-\frac{q}{q-2}}&\int_{a_n^j-R\lambda_n^{-\frac12}}^{a_n^j+R\lambda_n^{-\frac12}}u_n^q(r)dr=\int_{-R}^Rv_n^q(s)ds\to \int_{-R}^R\phi^q;\\
    \lambda_n^{\frac12-\frac{q}{q-2}}&\int_{a_n^j-R\lambda_n^{-\frac12}}^{a_n^j+R\lambda_n^{-\frac12}}\lambda_nu_n^2(r)dr=\int_{-R}^Rv_n^2(s)ds\to \int_{-R}^R\phi^2;\\
    \lambda_n^{\frac12-\frac{q}{q-2}}&\int_{a_n^j-R\lambda_n^{-\frac12}}^{a_n^j+R\lambda_n^{-\frac12}}|V(r)u_n^2(r)|dr\leq\lambda_n^{-1}\|V\|_\infty\int_{-R}^Rv_n^2(s)ds\to 0;\\
    \lambda_n^{\frac12-\frac{q}{q-2}}&\int_{a_n^j-R\lambda_n^{-\frac12}}^{a_n^j+R\lambda_n^{-\frac12}}\left|\frac{u_n^2(r)}{r^2}\right|dr\leq\frac{\lambda_n^{-1}}{(a_n^j-R\lambda_n^{-\frac12})^2}\int_{-R}^Rv_n^2(s)ds=\frac{1}{(\lambda_n^\frac12a_n^j-R)^2}\int_{-R}^Rv_n^2(s)ds\to 0\\
\end{split}
\end{equation*}
where $\phi$ is the unique solution to \eqref{Eq scaled convergence sphere}. We now turn to the estimates in ${I_{n}}\cap \mathcal{A}_{R,n}$. For some constant $C>0$ we have
    \begin{align*}
    \lambda_n^{\frac{1}{2}-\frac{q}{q-2}}\int_{\mathcal{A}_{R,n} \cap I_n}u_n^q(r) dr &\leq \lambda_n^{\frac{1}{2}-\frac{q}{q-2}}\epsilon^q\lambda_n^{\frac{q}{q-2}}e^{q\gamma R}\int_{\mathcal{A}_{R,n}}\left(\sum_{j=1}^Me^{-\gamma\lambda_n^\frac12|r-a_n^{j}|}\right)^q dr\nonumber\\
    &\leq C \lambda_n^{\frac{1}{2}}\epsilon^q e^{q\gamma R}\sum_{j=1}^M \int_{\complement_{\R^N} (A(a_n^{i},R\lambda_n^{-{1}/{2}}))}e^{-q\gamma\lambda_n^\frac12|r-a_n^{j}|}dr\nonumber\\
    &\leq C \epsilon^q e^{q\gamma R} \sum_{j=1}^M \int_{R}^{+\infty}e^{-q\gamma y}dy\nonumber\\
    &\leq \epsilon_1(R),\label{eq int estimate arn finite}
\end{align*}
where $\epsilon_1(R) >0$ can be made arbitrarily small by taking $R>0$ larger.
Similarly,
\begin{align*}
    \lambda_n^{\frac{1}{2}-\frac{q}{q-2}}&\int_{\mathcal{A}_{R,n}}\lambda_nu_n^2(r)dr\leq \epsilon_2(R),\\
    \lambda_n^{\frac{1}{2}-\frac{q}{q-2}}&\int_{\mathcal{A}_{R,n}}|V(r)|u_n^2(r)dr\leq \|V\|_\infty\lambda_n^{\frac{1}{2}-\frac{q}{q-2}}\int_{\mathcal{A}_{R,n}}u_n^2(r)dr\to 0.
\end{align*}
Also, since on $I_n$, $r \geq \frac{1}{2}\underline{K} a_n^i \geq \lambda_n^{- \frac{1}{2}}$ for all $n \in \N$ sufficiently large we get that
\begin{equation*}
\lambda_n^{\frac{1}{2}-\frac{q}{q-2}} \int_{\mathcal{A}_{R,n}\cap I_n}\frac{u_n^2(r)}{r^2}dr \leq \lambda_n^{\frac{1}{2}-\frac{q}{q-2}} \int_{\mathcal{A}_{R,n}}\lambda_nu_n^2(r)dr\leq \epsilon_2(R).
\end{equation*}
In all, multiplying \eqref{eq pozohaev pseudo} with $\lambda_n^{\frac{1}{2}-\frac{q}{q-2}}$, taking the limit as $n\to \infty$ and furthermore taking the limit as $R\to +\infty$, we conclude that for every $\epsilon>0$
$$|{J}_i|\cdot\left|\left(N-\frac32-\frac1{q}\right)\int_\R \phi^q+\left(2-N\right)\int_\R  \phi^2\right|\leq \epsilon $$
where $|J_i|$ is the number of elements in the set. However, since $\phi$ satisfies \eqref{Eq scaled convergence sphere} we may implement \eqref{norm relation sphere} to obtain
\begin{align}
    \epsilon&\geq |{J}_i|\cdot\left|\left[\left(N-\frac32-\frac1{q}\right)\int_\R \phi^q+\left(2-N\right)\int_\R  \phi^2\right]\right|\nonumber\\&=|{J}_i|\cdot \left|\left(N-\frac32-\frac1{q}\right)\left(\frac{2q}{q+2}\right)+\left(2-N\right)\right|\int_\R  \phi^2\nonumber
    \\&={|{J}_i|\cdot\left(\frac{(N-1)(q-2)}{q+2}\right)\int_\R  \phi^2}.\nonumber
\end{align}
This forces $J_i=\emptyset$ so, having chosen an arbitrary $i \in \{1, \cdots, M\}$ such that $\lambda_n^\frac12 a_n^i \to +\infty$ we have arrived at the desired conclusion that \eqref{goinfinity} may not happen.

\textit{Step 2: Discarding possibility \eqref{eq: blow up origin}.}  From the previous step we have the existence of a unique sequence $\{a_n\}$ of blow-up local maxima satisfying $a_n\lambda_n^{1/2}\to 0$, and, for $n$ sufficiently large,
$$u_n(a_n)=
\max\{u_n(x)\colon x\in B(0,a_n+R_n\lambda_n^{-\frac{1}{2}})\}
$$
  for some sequence $R_n\to +\infty$. In order to now discard this final possibility we will now focus on estimating 
  \begin{equation*}\label{contradiction integral}
      \mathscr{I}:=\lambda_n^{\frac{N}{2}-\frac{2}{q-2}}\int_{\R^N \backslash B(0, a_n +R\lambda_n^{-\frac12})}|u_n|^2 dx
  \end{equation*}
  in two different ways. First, fix $\epsilon>0$ and take $n\geq N_\epsilon$ and $R\geq R_\epsilon$ as in Theorem \ref{theo sec blow}. Using  the fact that the there is only one sequence of maxima, if $n$ is sufficiently large then \eqref{eq far exponential decay} becomes 
  $$u_n(r)\leq\epsilon\lambda_n^{\frac{1}{q-2}}e^{\gamma R}e^{-\gamma\lambda_n^\frac12|r-a_n|}\quad \text{for }\quad r\geq a_n+R\lambda_n^{-\frac12}.$$
Now, since $a_n\to 0$, we may assume that $a_n<1$ for $n \in \N$ sufficiently large and thus,
  \begin{align}
       \mathscr{I}
       &\leq \lambda_n^{\frac{N}{2}-\frac{2}{q-2}}\epsilon^2\lambda_n^{\frac{2}{q-2}}e^{2\gamma R} \int_{\R^N \backslash B(0, a_n +R\lambda_n^{-\frac12})}e^{-2\gamma\lambda_n^\frac12(|x|-a_n|)} dx \nonumber\\
       &= \lambda_n^{\frac{N}{2}}\epsilon^2 e^{s\gamma R}\int_{a_n+R\lambda_n^{-\frac12}}^{+\infty}r^{N-1}e^{-2\gamma\lambda_n^\frac12(r-a_n)}dr\nonumber\\
    &=\lambda_n^{\frac{N}2} \epsilon^2 e^{2\gamma R}\int_{R\lambda_n^{-\frac12}}^{+\infty}(y+a_n)^{N-1}\lambda_n^{-\frac{N-1}{2}}e^{-2\gamma\lambda_n^\frac12y}dy\nonumber\\
    &\leq \epsilon^2 e^{s\gamma R} \int_{R}^{+\infty}(z+1)^{N-1}e^{-2\gamma z}dz \label{eq:decay in Arn}
  \end{align}
	so \eqref{contradiction integral} remains bounded as $n\to \infty$ with $R$ fixed. On the other hand,
  \begin{align}
      \mathscr{I} &=\lambda_n^{\frac{N}{2}- \frac{2}{q-2}} \left(\int_{\R^N} u_n^2 dx -\int_{B(0,a_n+R\lambda_n^{-\frac12})} u_n^2 dx \right)\nonumber\\&\geq  \lambda_n^{\frac{N}{2}- \frac{2}{q-2}}\mu_1^2 -\lambda_n^{\frac{N}{2}- \frac{2}{q-2}} \int_{B(0,a_n+R\lambda_n^{-\frac12})} u_n^2 dx \to + \infty \label{eq:non decay in Arn}
      \end{align}
		
			The divergence of the term \eqref{eq:non decay in Arn} 
			follows directly from \eqref{eq:mass-convergence origin}, used with $s=2$.
  In all, taking the limit as ${n\to\infty}$ and keeping $R$ fixed results in a contradiction between \eqref{eq:decay in Arn} and  \eqref{eq:non decay in Arn}. This implies that the sequence of Lagrange multipliers $\{\lambda_n\}$ must, in all cases, remain bounded since otherwise there would exist sequences of blow point which, as we have seen, lead to a contradiction. The proof is now finished.
\end{proof}

\renewcommand{\bibname}{References}
\bibliographystyle{plain}
\bibliography{References_Louis_Pablo_et_these}

\begin{thebibliography}{10}

\bibitem{AmMaNi03-1}
A.~Ambrosetti, A.~Malchiodi, and W.-M. Ni.
\newblock Singularly perturbed elliptic equations with symmetry: existence of
  solutions concentrating on spheres. {I}.
\newblock {\em Comm. Math. Phys.}, 235(3):427--466, 2003.

\bibitem{BaLi}
A.~Bahri and P.-L. Lions.
\newblock Solutions of superlinear elliptic equations and their {M}orse
  indices.
\newblock {\em Comm. Pure Appl. Math.}, 45(9):1205--1215, 1992.

\bibitem{BaJe18}
T.~Bartsch and L.~Jeanjean.
\newblock Normalized solutions for nonlinear {S}chr\"{o}dinger systems.
\newblock {\em Proc. Roy. Soc. Edinburgh Sect. A}, 148(2):225--242, 2018.

\bibitem{BMRV}
T.~Bartsch, R.~Molle, M.~Rizzi, and G.~Verzini.
\newblock Normalized solutions of mass supercritical {S}chr\"{o}dinger
  equations with potential.
\newblock {\em Comm. Partial Differential Equations}, 46(9):1729--1756, 2021.

\bibitem{BQZ}
T.~Bartsch, S.~Qi, and W.~Zou.
\newblock Normalized solutions to {S}chr\"{o}dinger equations with potential
  and inhomogeneous nonlinearities on large smooth domains.
\newblock {\em Math. Ann.}, 390(3):4813--4859, 2024.

\bibitem{BaSo17}
T.~Bartsch and N.~Soave.
\newblock A natural constraint approach to normalized solutions of nonlinear
  {S}chr\"{o}dinger equations and systems.
\newblock {\em J. Funct. Anal.}, 272(12):4998--5037, 2017.

\bibitem{BaSo19}
T.~Bartsch and N.~Soave.
\newblock Multiple normalized solutions for a competing system of
  {S}chr\"{o}dinger equations.
\newblock {\em Calc. Var. Partial Differential Equations}, 58(1):Paper No. 22,
  24, 2019.

\bibitem{BeBoJeVi17}
J.~Bellazzini, N.~Boussa\"{\i}d, L.~Jeanjean, and N.~Visciglia.
\newblock Existence and stability of standing waves for supercritical {NLS}
  with a partial confinement.
\newblock {\em Comm. Math. Phys.}, 353(1):229--251, 2017.

\bibitem{BeLi1}
H.~Berestycki and P.-L. Lions.
\newblock Nonlinear scalar field equations. {I}. {E}xistence of a ground state.
\newblock {\em Arch. Rational Mech. Anal.}, 82(4):313--345, 1983.

\bibitem{BeLi2}
H.~Berestycki and P.-L. Lions.
\newblock Nonlinear scalar field equations. {II}. {E}xistence of infinitely
  many solutions.
\newblock {\em Arch. Rational Mech. Anal.}, 82(4):347--375, 1983.

\bibitem{BiMe21}
B.~Bieganowski and J.~Mederski.
\newblock Normalized ground states of the nonlinear {S}chr\"{o}dinger equation
  with at least mass critical growth.
\newblock {\em J. Funct. Anal.}, 280(11):Paper No. 108989, 26, 2021.

\bibitem{BoChJeSo23}
J.~Borthwick, X.~Chang, L.~Jeanjean, and N.~Soave.
\newblock Normalized solutions of {$L^2$}-supercritical {NLS} equations on
  noncompact metric graphs with localized nonlinearities.
\newblock {\em Nonlinearity}, 36(7):3776--3795, 2023.

\bibitem{BoChJeSo24}
J.~Borthwick, X.~Chang, L.~Jeanjean, and N.~Soave.
\newblock Bounded {P}alais-{S}male sequences with {M}orse type information for
  some constrained functionals.
\newblock {\em Trans. Amer. Math. Soc.}, 377(6):4481--4517, 2024.

\bibitem{CaGaJeTr25}
P.~Carrillo, C.~De~Coster, D.~Galant, L.~Jeanjean, and C.~Troestler.
\newblock Blow-up analysis and a priori bounds for {NLS} equations on metric
  graphs.
\newblock {\em arXiv.2605.02879}, 2026.

\bibitem{ChJeSo24}
X.~Chang, L.~Jeanjean, and N.~Soave.
\newblock Normalized solutions of {$L^2$}-supercritical {NLS} equations on
  compact metric graphs.
\newblock {\em Ann. Inst. H. Poincar\'{e} C Anal. Non Lin\'{e}aire},
  41(4):933--959, 2024.

\bibitem{dPKoWe06}
M.~del Pino, M.~Kowalczyk, and J.-C. Wei.
\newblock Concentration on curves for nonlinear {S}chr\"{o}dinger equations.
\newblock {\em Comm. Pure Appl. Math.}, 60(1):113--146, 2007.

\bibitem{DHR}
O.~Druet, E.~Hebey, and F.~Robert.
\newblock {\em Blow-up theory for elliptic {PDE}s in {R}iemannian geometry},
  volume~45 of {\em Mathematical Notes}.
\newblock Princeton University Press, Princeton, NJ, 2004.

\bibitem{EMSS}
P.~Esposito, G.~Mancini, S.~Santra, and P.~N. Srikanth.
\newblock Asymptotic behavior of radial solutions for a semilinear elliptic
  problem on an annulus through {M}orse index.
\newblock {\em J. Differential Equations}, 239(1):1--15, 2007.

\bibitem{EsPe11}
P.~Esposito and M.~Petralla.
\newblock Pointwise blow-up phenomena for a {D}irichlet problem.
\newblock {\em Comm. Partial Differential Equations}, 36(9):1654--1682, 2011.

\bibitem{FrLaWe23}
R.~L. Frank, A.~Laptev, and T.~Weidl.
\newblock {\em Schr\"{o}dinger operators: eigenvalues and {L}ieb-{T}hirring
  inequalities}, volume 200 of {\em Cambridge Studies in Advanced Mathematics}.
\newblock Cambridge University Press, Cambridge, 2023.

\bibitem{GiNiNi81}
B.~Gidas, W.-M. Ni, and L.~Nirenberg.
\newblock Symmetry of positive solutions of nonlinear elliptic equations in
  {${\bf R}^{n}$}.
\newblock In {\em Mathematical analysis and applications, {P}art {A}}, volume~7
  of {\em Adv. Math. Suppl. Stud.}, pages 369--402. Academic Press, New
  York-London, 1981.

\bibitem{GiTr}
D.~Gilbarg and N.~Trudinger.
\newblock {\em Elliptic partial differential equations of second order}, volume
  224.
\newblock Springer, 1977.

\bibitem{GoJe18}
T.~Gou and L.~Jeanjean.
\newblock Multiple positive normalized solutions for nonlinear
  {S}chr\"{o}dinger systems.
\newblock {\em Nonlinearity}, 31(5):2319--2345, 2018.

\bibitem{HiSi}
P.~Hislop and M.~I. Sigal.
\newblock {\em Introduction to spectral theory: With applications to
  Schr{\"o}dinger operators}, volume 113.
\newblock Springer Science and Business Media, 2012.

\bibitem{IkMi20}
N.~Ikoma and Y.~Miyamoto.
\newblock Stable standing waves of nonlinear {S}chr\"{o}dinger equations with
  potentials and general nonlinearities.
\newblock {\em Calc. Var. Partial Differential Equations}, 59(2):Paper No. 48,
  20 pp,, 2020.

\bibitem{IkTa19}
N.~Ikoma and K.~Tanaka.
\newblock A note on deformation argument for {$L^2$} normalized solutions of
  nonlinear {S}chr\"{o}dinger equations and systems.
\newblock {\em Adv. Differential Equations}, 24(11-12):609--646, 2019.

\bibitem{IkYa24}
N.~Ikoma and M.~Yamanobe.
\newblock The existence and multiplicity of {$L^2$}-normalized solutions to
  nonlinear {S}chr\"{o}dinger equations with variable coefficients.
\newblock {\em Adv. Nonlinear Stud.}, 24(2):477--509, 2024.

\bibitem{Je97}
L.~Jeanjean.
\newblock Existence of solutions with prescribed norm for semilinear elliptic
  equations.
\newblock {\em Nonlinear Anal.}, 28(10):1633--1659, 1997.

\bibitem{Je99}
L.~Jeanjean.
\newblock On the existence of bounded {P}alais-{S}male sequences and
  application to a {L}andesman-{L}azer-type problem set on {$\mathbb{R}^N$}.
\newblock {\em Proc. Roy. Soc. Edinburgh Sect. A}, 129(4):787--809, 1999.

\bibitem{JeLe22}
L.~Jeanjean and T.~T. Le.
\newblock Multiple normalized solutions for a {S}obolev critical
  {S}chr\"{o}dinger equation.
\newblock {\em Math. Ann.}, 384(1-2):101--134, 2022.

\bibitem{JeLu20}
L.~Jeanjean and S.~S. Lu.
\newblock A mass supercritical problem revisited.
\newblock {\em Calc. Var. Partial Differential Equations}, 59(5):Paper No. 174,
  43, 2020.

\bibitem{JeSo25}
L.~Jeanjean and L.~Song.
\newblock Multiplicity result for a mass supercritical {NLS} with a partial
  confinement.
\newblock {\em SIAM J. Math. Anal.}, 57(6):6709--6730, 2025.

\bibitem{LaMo24}
S.~Lancelotti and R.~Molle.
\newblock Normalized positive solutions for {S}chr\"{o}dinger equations with
  potentials in unbounded domains.
\newblock {\em Proc. Roy. Soc. Edinburgh Sect. A}, 154(5):1518--1551, 2024.

\bibitem{LiSo24}
J.~Liang and L.~Song.
\newblock On radial positive normalized solutions of the nonlinear
  {S}chr\"{o}dinger equation in an annulus.
\newblock {\em NoDEA Nonlinear Differential Equations Appl.}, 31(2):Paper No.
  20, 14, 2024.

\bibitem{Lions1984a}
P.-L. Lions.
\newblock The concentration-compactness principle in the calculus of
  variations. {T}he locally compact case. {I}.
\newblock {\em Ann. Inst. H. Poincar\'{e} Anal. Non Lin\'{e}aire},
  1(2):109--145, 1984.

\bibitem{Lions1984b}
P.-L. Lions.
\newblock The concentration-compactness principle in the calculus of
  variations. {T}he locally compact case. {II}.
\newblock {\em Ann. Inst. H. Poincar\'{e} Anal. Non Lin\'{e}aire},
  1(4):223--283, 1984.

\bibitem{MRV}
R.~Molle, G.~Riey, and G.~Verzini.
\newblock Normalized solutions to mass supercritical {S}chr\"{o}dinger
  equations with negative potential.
\newblock {\em J. Differential Equations}, 333:302--331, 2022.

\bibitem{NoTaVe15}
B.~Noris, H.~Tavares, and G.~Verzini.
\newblock Stable solitary waves with prescribed {$L^2$}-mass for the cubic
  {S}chr\"{o}dinger system with trapping potentials.
\newblock {\em Discrete Contin. Dyn. Syst.}, 35(12):6085--6112, 2015.

\bibitem{PPVV}
B.~Pellacci, A.~Pistoia, G.~Vaira, and G.~Verzini.
\newblock Normalized concentrating solutions to nonlinear elliptic problems.
\newblock {\em J. Differential Equations}, 275:882--919, 2021.

\bibitem{PiVe17}
D.~Pierotti and G.~Verzini.
\newblock Normalized bound states for the nonlinear {S}chr\"{o}dinger equation
  in bounded domains.
\newblock {\em Calc. Var. Partial Differential Equations}, 56(5):Paper No. 133,
  27, 2017.

\bibitem{ReSi78}
M.~Reed and B.~Simon.
\newblock {\em Methods of modern mathematical physics. {IV}. {A}nalysis of
  operators}.
\newblock Academic Press [Harcourt Brace Jovanovich, Publishers], New
  York-London, 1978.

\bibitem{So20}
N.~Soave.
\newblock Normalized ground states for the {NLS} equation with combined
  nonlinearities.
\newblock {\em J. Differential Equations}, 269(9):6941--6987, 2020.

\bibitem{So20critical}
N.~Soave.
\newblock Normalized ground states for the {NLS} equation with combined
  nonlinearities: the {S}obolev critical case.
\newblock {\em J. Funct. Anal.}, 279(6):108610, 43, 2020.

\bibitem{St81}
C.~A. Stuart.
\newblock Bifurcation from the continuous spectrum in the {$L^2$}-theory of
  elliptic equations on {${\bf R}^n$}.
\newblock In {\em Recent methods in nonlinear analysis and applications
  ({N}aples, 1980)}, pages 231--300. Liguori, Naples, 1981.

\bibitem{St82}
C.~A. Stuart.
\newblock Bifurcation for {D}irichlet problems without eigenvalues.
\newblock {\em Proc. London Math. Soc. (3)}, 45(1):169--192, 1982.

\bibitem{Ta}
T.~Tao.
\newblock {\em Nonlinear dispersive equations}, volume 106 of {\em CBMS
  Regional Conference Series in Mathematics}.
\newblock Published for the Conference Board of the Mathematical Sciences,
  Washington, DC; by the American Mathematical Society, Providence, RI, 2006.
\newblock Local and global analysis.

\bibitem{VeYu25}
G.~Verzini and J.~Yu.
\newblock Normalized solutions for the nonlinear {S}chr\"{o}dinger equation
  with potential: the purely {S}obolev critical case.
\newblock {\em Calc. Var. Partial Differential Equations}, 65(3):80, 2026.

\bibitem{WeWu22}
J.~Wei and Y.~Wu.
\newblock Normalized solutions for {S}chr\"{o}dinger equations with critical
  {S}obolev exponent and mixed nonlinearities.
\newblock {\em J. Funct. Anal.}, 283(6):Paper No. 109574, 46, 2022.

\end{thebibliography}
\vspace{0.25cm}

\end{document}